\documentclass{ws-ijac}

\usepackage{amsmath,amssymb}
\usepackage[dvipsnames]{xcolor}
\usepackage{tocloft}

\newtheorem{algorithm}[theorem]{Algorithm}

\newtheorem{assumption}[theorem]{Assumption}

\makeatletter
\def\@listI{\leftmargin\leftmargini
            \listparindent\parindent
            \parsep \z@\labelsep.5em
            \topsep 0.5\p@ \@plus1.5\p@
            \itemsep0\p@}
\makeatother

\def\OO{{\mathcal{O}}}
\def\QQ{\mathbb{Q}}
\def\ZZ{\mathbb{Z}}
\def\X{{\mathbb{X}}}
\def\F{{\mathcal{F}}}
\def\H{{\mathcal{H}}}
\def\AA{\mathbb{A}}
\def\M{\mathfrak{M}}
\def\Q{{\mathfrak{Q}}}
\def\m{\mathfrak{m}}
\def\q{\mathfrak{q}}
\def\A{\mathcal{A}}
\def\B{\mathbb{B}}

\let\epsilon=\varepsilon

\def\phi{{\varphi}}
\let\Psi=\varPsi
\let\Phi=\varPhi
\let\theta=\vartheta

\DeclareSymbolFont{CMMfont}{OML}{cmm}{m}{it}
\DeclareMathSymbol{\Varrho}{3}{CMMfont}{37}
\def\rho{{\mathop{\Varrho}\,}}

\def\LT{\mathop{\rm LT}\nolimits}

\def\LF{\mathop{\rm LF}\nolimits}
\def\DF{\mathop{\rm DF}\nolimits}
\def\HF{\mathop{\rm HF}\nolimits}
\def\HFa{\mathop{\rm HF}\nolimits^a}
\def\ri{\mathop{\rm ri}\nolimits}

\def\Mat{\mathop{\rm Mat}\nolimits}

\def\Spec{\mathop{\rm Spec}\nolimits}
\def\Hilb{\mathop{\rm Hilb}\nolimits}

\def\Rad{\mathop{\rm Rad}\nolimits}

\def\Col{\mathop{\rm Col}\nolimits}
\def\gr{\mathop{\rm gr}\nolimits}
\def\grF{{\mathop{\rm gr}\nolimits_{\mathcal{F}}}}
\def\ord{\mathop{\rm ord}\nolimits}
\def\sepdeg{\mathop{\rm sepdeg}\nolimits}

\def\Hbar{\overline{\mathcal{H}}}
\def\adj{\mathop{\rm adj}\nolimits}

\def\Fbar{{\overline{\mathcal{F}}}}

\def\BO{\mathbb{B}_{\mathcal{O}}}
\def\BOdf{\mathbb{B}_{\mathcal{O}}^{\rm df}}
\def\BOhom{\mathbb{B}_{\mathcal{O}}^{\rm hom}}

\def\LGor{\mathop{\rm LGor}\nolimits}
\def\NonLGor{\mathop{\rm NonLGor}\nolimits}
\def\CB{\mathop{\rm CB}\nolimits}
\def\NonCB{\mathop{\rm NonCB}\nolimits}
\def\SCB{\mathop{\rm SCB}\nolimits}
\def\NonSCB{\mathop{\rm NonSCB}\nolimits}
\def\SGor{\mathop{\rm SGor}\nolimits}
\def\NonSGor{\mathop{\rm NonSGor}\nolimits}
\def\SCI{\mathop{\rm SCI}\nolimits}
\def\NonSCI{\mathop{\rm NonSCI}\nolimits}

\let\To=\longrightarrow

\def\hom{^{\rm hom}}
\def\df{^{\rm df}}
\def\tr{^{\,\rm tr}}
\def\tfrac #1#2{{\textstyle\frac{#1}{#2}}}

\def\cocoa{\mbox{\rm
  C\kern-.13em o\kern-.07 em C\kern-.13em o\kern-.15em A}}
\def\apcocoa{\mbox{\rm
A\kern-0.13em p\kern -0.07em C\kern-.13em o\kern-.07 em C\kern-.13em
o\kern-.15em A}}

\begin{document}
\markboth{Kreuzer, Long and Robbiano}
{Computing Subschemes of the Border Basis Scheme}

\title{COMPUTING SUBSCHEMES OF THE BORDER BASIS SCHEME}

\author{MARTIN KREUZER}
\address{Fakult\"at f\"ur Informatik und Mathematik, Universit\"at
Passau, D-94030 Passau, Germany\\
\email{Martin.Kreuzer@uni-passau.de}
 }

\author{LE NGOC LONG}
\address{Fakult\"at f\"ur Informatik und Mathematik, Universit\"at
Passau, D-94030 Passau, Germany and Department of Mathematics, 
University of Education - Hue University, 34 Le Loi, Hue, Vietnam\\
\email{lelong@hueuni.edu.vn}  }

\author{LORENZO ROBBIANO}
\address{Dipartimento di Matematica, Universit\`a di Genova,
Via Dodecaneso 35, I-16146 Genova, Italy\\ 
\email{lorobbiano@gmail.com}  }

\maketitle

\begin{abstract}
{\it Abstract:}
A good way of parametrizing 0-dimensional schemes in an affine 
space~$\AA_K^n$ has been developed in the last 20 years using border 
basis schemes.
Given a multiplicity~$\mu$, they provide an open covering of the Hilbert 
scheme $\Hilb^\mu(\AA^n_K)$ and can be described by 
easily computable quadratic equations. A natural question arises on how
to determine loci which are contained in border basis schemes and 
whose rational points represent 0-dimensional $K$-algebras sharing a given property.
The main focus of this paper is on giving effective answers to this general
problem. The properties considered here are the locally Gorenstein, strict Gorenstein,
strict complete intersection, Cayley-Bacharach, and strict Cayley-Bacharach properties.
The key characteristic of our approach is that we describe these loci by exhibiting 
explicit algorithms to compute their defining ideals. 
All results are illustrated by non-trivial, concrete examples.
\end{abstract}




\markboth{Kreuzer, Long and Robbiano}
{Subschemes of the Border Basis Scheme}

\bigskip
%
%

\section{Introduction}

In~\cite{Gro}, A.~Grothendieck defined the Hilbert scheme
$\Hilb^\mu(\AA^n_K)$ which parametrizes all 0-dimensional subschemes
of length~$\mu$ of~$\mathbb{A}^n_K$ as a quotient of a suitable Grassmannian variety.
If we turn this construction into a computable presentation, its 
defining relations yield a rather large and unwieldy system of equations
(see, for instance, \cite{IK}).
A more convenient way of parametrizing 0-dimensional schemes in~$\mathbb{A}^n_K$,
or equivalently, of 0-dimensional affine algebras over a field~$K$, 
has been developed in the last 20 years, namely the border basis scheme
(see \cite{Hai}, \cite{Hui1}, \cite{KR4}, \cite{KR5}, and~\cite{Rob}).

The border basis schemes corresponding to order ideals of length~$\mu$
form an open covering of the Hilbert scheme $\Hilb^\mu(\mathbb{A}^n_K)$,
and they have simple, explicit presentations involving easily describable
quadratic equations.
Having such explicit descriptions of large families of 0-dimensional
schemes, it is a natural question to ask for similar descriptions
of the subsets given by all \hbox{0-di}\-men\-sional schemes having certain
geometric or algebraic properties. This is the task that we 
tackle here, and we treat algebraic properties such as the
locally Gorenstein, strict Gorenstein and strict complete intersection
properties, as well as geometric properties such as the Cayley-Bacharach
and strict Cayley-Bacharach properties. Notice that, for instance, although the
locally Gorenstein locus of the Hilbert scheme is actively researched (e.g., 
see~\cite{BCR19}, \cite{CJN15}, \cite{CN09}, \cite{CN11}), the simple 
explicit defining equations computed here appear to be new.

In more detail, we proceed as follows. In Sections~2 and~3, we
recall some basic material about border bases and border basis schemes,
respectively. Let $\OO=\{t_1,\dots,t_\mu\}$ be an order ideal of
terms in $K[x_1,\dots,x_n]$, i.e., a factor-closed finite set of
power products of indeterminates. We define its border
$\partial\OO= \bigl( \bigcup_{i=1}^n x_i\OO \bigr) \setminus \OO
= \{b_1,\dots,b_\nu\}$ and the generic $\OO$-border prebasis
$G=\{g_1,\dots,g_\nu\}$, where $g_j=b_j -\sum_{i=1}^\mu c_{ij}t_i$
and where $C=\{c_{ij}\}$ is a set of new indeterminates.
Then the border basis scheme~$\BO$ is the subscheme of 
$\mathbb{A}^{\mu\nu}_K \cong \Spec(K[C])$ defined by the vanishing
of the entries of the commutators of the generic multiplication
matrices for $K[C][x_1,\dots,x_n]/\langle G\rangle$.
Notice that the scheme~$\BO$ may be non-reduced and can have
irreducible components of higher dimension than the expected
dimension $n\mu$ (see~\cite{KR4}, Example 5.6). Nevertheless it is
defined by a set of nice and easy quadratic polynomials
in the indeterminates $c_{ij}$.

In Section~4 we start by looking at the subscheme of~$\BO$ defined
by the property that the affine coordinate ring~$R_\X$ of the scheme~$\X$
represented by a $K$-rational point of~$\BO$ is locally Gorenstein.
This ring-theoretic property has been characterized in~\cite{KLR1}
by the condition that a certain determinant of a matrix constructed 
from the multiplication matrices of~$R_\X$ with respect to the basis~$\OO$
is non-zero. The global version of this characterization yields
Algorithm~\ref{alg:GorLoc}, where the equations defining the locally Gorenstein
locus in~$\BO$ are calculated explicitly.

Other loci require us to stratify the border basis scheme before we can
compute their defining ideals. A first important step in this direction
is taken in Section~5, where we define and study the degree filtered 
border basis scheme~$\BOdf$. This is the subscheme of~$\BO$ whose
$K$-rational points represent schemes~$\X$ such that~$\OO$ is a degree
filtered $K$-basis of~$R_\X$, i.e., such that the elements of degree $\le i$
in~$\OO$ form a $K$-basis of the $i$-th part of the degree filtration of~$R_\X$
for every $i\ge 0$. The scheme~$\BOdf$ is the basic Hilbert stratum
of~$\BO$, and the computation of various loci in other Hilbert strata
will be reduced to this case later in the paper.

Sections~6, 7, and~8 provide several algorithms which are central to this paper.
Firstly, we compute the Cayley-Bacharach locus in~$\BOdf$ in Algorithm~\ref{alg:CBLocDF}.
Here we use the results of~\cite{KLR1} again and generalize them to 
the universal border basis family. Recall that the Cayley-Bacharach property
of a 0-dimensional scheme is a geometric property which has been generalized
multiple times in the history of Mathematics (see the introduction of~\cite{KLR1})
and is used here in its most general form: over an arbitrary base field~$K$, and
for an arbitrary 0-dimensional scheme~$\X$ in~$\mathbb{A}^n_K$.

Secondly, we compute the strict Cayley-Bacharach locus in~$\BOdf$ 
in Algorithm~\ref{alg:SCBLocDF}. The adjective ``strict'' indicates that we are looking
for the locus in~$\BOdf$ whose \hbox{$K$-ra}\-tio\-nal points represent schemes~$\X$
such that the graded ring $\grF(R_\X)$ of~$R_\X$ with respect to the degree 
filtration has the Cayley-Bacharach property. Geometrically, this means that
also the tip of the affine cone over~$\X$ has this property, and in
projective geometry, the corresponding adjective is ``arithmetical''.
Algorithm~\ref{alg:SCBLocDF} is based on the
homogeneous border basis scheme~$\BOhom$, recalled in Section~5, which parametrizes
the rings~$\grF(R_\X)$, and also on Algorithm~\ref{alg:CheckSCBS} which
allows us to check the strict Cayley-Bacharach property for a single
scheme~$\X$ using its multiplication matrices.

Thirdly, the strict complete intersection locus in~$\BOdf$ is computed
in Algorithm~\ref{alg:SCILocDF}. Again, the adjective ``strict'' refers
to the ring $\grF(R_\X)$ being a complete intersection, or, equivalently,
that $R_\X = K[x_1,\dots,x_n]/\langle f_1,\dots,f_n\rangle$ where the degree 
forms $\{\DF(f_1),\dots,\DF(f_n\}\}$ form a (homogeneous) regular sequence. 
For one scheme~$\X$, an algorithm for checking this property using the 
multiplication matrices of~$R_\X$ was developed in~\cite{KLR2}. 
Here we extend this method to the universal border basis family and get 
Algorithm~\ref{alg:SCBLocDF}.

The further parts of the paper focus on using these results to describe
the corresponding loci in all of~$\BO$. To this end, we introduce
and compute the Hilbert stratification of~$\BO$ in Section~9. Given an
admissible Hilbert function~$\H$, we first compute the closed 
subscheme~$\BO(\Hbar)$ of~$\BO$ whose $K$-rational points represent
schemes such that their affine Hilbert function is dominated by~$\H$, 
and then the open subscheme~$\BO(\H)$ of~$\BO(\Hbar)$ which is the locus 
where the affine Hilbert function is exactly~$\H$. 

However, for several applications, we even have to fix a degree filtered
basis of the ring~$R_\X$. This is achieved in Section~10 by covering
$\BO(\H)$ with open subschemes $\BO^{\rm dfb}(\OO')$, called $\OO'$-DFB subschemes,
such that the order ideal~$\OO'$ is a degree filtered $K$-basis of~$R_\X$
for all schemes~$\X$ represented by $K$-rational points of $\BO^{\rm dfb}(\OO')$. 
The ideals describing the various schemes $\BO^{\rm dfb}(\OO')$
are computed in Algorithm~\ref{alg:DFB}. Then the key result for the
last parts of the paper is Algorithm~\ref{alg:Combine}. It allows us
to combine the ideals describing certain loci in the individual schemes
$\BO^{\rm dfb}(\OO')$ to ideals describing them in~$\BO(\H)$, and the underlying
base changes and morphisms of universal families are made explicit. 
Some computational simplifications are suggested in Remark~\ref{ImprovedCombine}.

The last three Sections~11, 12, and~13 apply the technique of Algorithm~\ref{alg:Combine}
and the respective methods to the degree filtered case.
In Section~11, we compute the Cayley-Bacharach locus of~$\BO(\H)$ 
(see Algorithm~\ref{alg:CBlocus}) and derive a method for calculating
the strict Gorenstein locus (see Corollary~\ref{alg:SGorInBOH}).
Next, Section~12 treats the strict Cayley-Bacharach locus of~$\BO(\H)$
via Algorithm~\ref{alg:SCBlocus}, and Section~13 
shows how to compute the strict complete intersection locus in~$\BO(\H)$
via Algorithm~\ref{alg:SCIlocus}.

All algorithms are illustrated by applying them to non-trivial, explicit
examples. They show that one can ``really do it'' in small cases, rather
than being able to ``do it in principle".
These examples were calculated by a package
for the computer algebra system \cocoa\ (see~\cite{CoCoA})
written by the second author, and available via the 
project web page (see~\cite{KLR3}).

Before starting with the paper properly, let us mention some general aspects of
the results and their presentation. We always work over an arbitrary 
base field~$K$. This entails that all results are characteristic-free.
Although we shall frequently talk about $K$-rational points of
certain schemes, such points may not exist over the given field.
Notice that we may enlarge~$K$ without changing the results
of the computations, because they are performed over~$K$.
Hence the claims remain true if we consider all $L$-rational points
for a field extension $L\supseteq K$, for instance, 
the algebraic closure $L=\overline{K}$, and any mention of $K$-rational points should
be read in this way.

In principle, all subschemes we consider should be equipped with the induced reduced
scheme structure. However, our algorithms produce ideals defining the desired subsets 
of the border basis scheme which are not necessarily radical ideals. Of course, if we replace
one ideal~$I$ by another ideal~$J$ such that $\Rad(I)=\Rad(J)$, 
then the ideal~$J$ defines the same set of $K$-rational points as~$I$. 
To simplify the exposition and the calculations, when we write that an ideal ``defines a subscheme'',
we really mean ``up to radical''. On a number of occasions, this freedom will be essential, 
since the calculation of the radical of an ideal in many indeterminates can be an unsurmountable 
burden.

Unless explicitly stated otherwise, we use the notation
and definitions of~\cite{KR1}, \cite{KR2}, and~\cite{KR3}.

\bigbreak
%
%

\section{Border Bases of Zero-Dimensional Ideals}
\label{Border Bases of Zero-Dimensional Ideals}

In the following we let $K$ be a field, let $P=K[x_1,\dots,x_n]$
be a polynomial ring over~$K$, let $\mathbb{T}^n = 
\{ x_1^{\alpha_1}\cdots x_n^{\alpha_n} \mid \alpha_i \ge 0\}$
be the monoid of terms in~$P$, and let~$I$ be a 0-dimensional ideal 
in~$P$. Border bases of~$I$ are defined as follows.

\begin{definition} Let $\mu\ge 1$.
\begin{enumerate}
\item[(a)] A set of terms $\OO=\{t_1,\dots,t_\mu\}$
is called an {\bf order ideal} in~$\mathbb{T}^n$ if
$t\in\OO$ implies that every term $t'\in\mathbb{T}^n$
which divides~$t$ is also contained in~$\OO$. 

\item[(b)] For an order ideal $\OO$, the set of terms
$\partial\OO = (x_1\OO \cup \cdots \cup x_n\OO)\setminus \OO$
is called the {\bf border} of~$\OO$. 

\item[(c)] Let $\OO=\{t_1,\dots,t_\mu\}$ be an order ideal
and $\partial \OO = \{b_1,\dots,b_\nu\}$ its border.
A set of polynomials $G=\{g_1,\dots,g_\nu\}$ in~$P$ is called
an {\bf $\OO$-border prebasis} if they are of the form
$g_j = b_j -\sum_{i=1}^\mu \gamma_{ij}\, t_i$ with
$\gamma_{ij}\in K$ for $i=1,\dots,\mu$ and $j=1,\dots,\nu$.

\item[(d)] An $\OO$-border prebasis $G\subset I$ is called an 
{\bf $\OO$-border basis} of~$I$ if the residue classes
of the terms in~$\OO$ from a $K$-basis of~$P/I$.
\end{enumerate}
\end{definition}

For more information about border bases, we refer to~\cite{KR2}, 
Section~6.4. In particular, we note that a border basis always generates
the ideal~$I$.

\begin{assumption}\label{orderOO}
For an order ideal $\OO=\{t_1,\dots,t_\mu\}$ in~$\mathbb{T}^n$,
we always assume that $\deg(t_1) \le \cdots \le \deg (t_\mu)$.
In particular, this implies that we have $t_1=1$.
\end{assumption}

Another way to view this setting is given by Algebraic Geometry.
Here we consider the 0-dimensional subscheme~$\X$ of~$\AA^n_K$
whose vanishing ideal is~$I$. We denote its affine coordinate ring $P/I$
by~$R_\X$ and write~$I_\X$ instead of~$I$.
Notice that we always choose a fixed embedding of~$\X$
into $\AA^n_K = \Spec(P)$ and also fix the coordinate system.
Clearly, the notion of a border basis of~$I_\X$ depends on these choices.
The vector space dimension of~$R_\X$ over~$K$ is finite.
It is sometimes called the {\bf length} of~$\X$ and 
will be denoted by~$\mu= \dim_K(R_\X)$.

Since we are keeping the coordinate system fixed
at all times, we have further invariants of~$\X$.
Recall that the {\bf degree filtration} $\widetilde{\F}
= (F_iP)_{i\in\ZZ}$ on~$P$ is given by $F_iP =
\{f\in P\setminus \{0\} \mid \deg(f)\le i\} \cup \{0\}$
for all $i\in\ZZ$. The induced filtration
$\F = (F_iR_\X)_{i\in\ZZ}$, where $F_i R_\X =
F_iP /(F_iP\cap I_\X)$, is called the {\bf degree filtration} 
on~$R_\X$. The degree filtration
on~$R_\X$ is increasing, exhaustive and {\bf orderly} in the
sense that every element $\bar{f}\in R_\X \setminus \{0\}$ has an
{\bf order} 
$$
\ord_\F(\bar{f}) \;=\; \min\{i\in\ZZ \mid \bar{f}\in F_i R_\X
\setminus F_{i-1}R_\X\}
$$
The degree filtration allows us to introduce the following
concepts.

\begin{definition}\label{def:affineHF}
Let~$\X$ be a 0-dimensional subscheme of~$\AA^n_K$
of length~$\mu$.
\begin{enumerate}
\item[(a)] The map $\HFa_\X: \ZZ \To \ZZ$ given by $i\mapsto 
\dim_K(F_i R_\X)$ is called the {\bf affine Hilbert function} of~$\X$.
It is a monotonously increasing function which satisfies 
$\HFa_\X(i)=\mu$ for $i\gg 0$. Here the number 
$$
\ri(R_\X)= \min \{i\in\ZZ \mid 
\HFa_\X(j)=\mu \hbox{\ for all\ }j\ge i\}
$$ 
is called the {\bf regularity index} of~$\X$.

\item[(b)] The first difference function $\Delta\HFa_\X (i) = 
\HFa_\X(i) - \HFa_\X(i-1)$ of~$\HFa_\X$ is called the
{\bf Castelnuovo function} of~$\X$, and $\Delta_\X = 
\Delta \HFa_\X(\ri(R_\X))$ is called the {\bf last difference} of~$\X$.

\item[(c)] Given an order ideal $\OO=\{t_1,\dots,t_\mu\}$ in~$\mathbb{T}^n$
and a number $i\ge 0$, we let $h_i=\# \, \{ j\in \{1,\dots,\mu\}\mid \deg(t_j)=i \}$. 
Then $\HF_\OO = (h_0,h_1,\dots)$ is called the {\bf Hilbert function} of~$\OO$,
and $\HFa_\OO = (h_0, h_0+h_1, h_0+h_1+h_2,\dots)$ is called the 
{\bf affine Hilbert function} of~$\OO$. In other words, we have
$\HF_{\OO}=\Delta \HFa_\OO$.
\end{enumerate}
\end{definition}

The affine Hilbert function of~$\X$ satisfies $\HFa_\X(i)=0$ for
$i<0$ and
$$
1 = \HFa_\X(0) < \HFa_\X(1) < \cdots < \HFa_\X(\ri(R_\X)) = \mu
$$
as well as $\HFa_\X(i)=\mu$ for $i\ge \ri(R_\X)$. 
Notice that, in general, the affine Hilbert function of~$\OO$
differs from the affine Hilbert function of~$\X$, even
if~$\OO$ represents a $K$-basis of~$R_\X$. We will come back to 
this point later (see Section~\ref{The Degree Filtered Border Basis Scheme BODF}).

One application of the degree filtration on~$R_\X$ is the
possibility of passing to the degree forms of polynomials
and to reduce many considerations to the homogeneous case.
Recall that the {\bf degree form} of a polynomial 
$f\in P\setminus \{0\}$ is its homogeneous component of highest
degree and is denoted by $\DF(f)$.
Given an ideal $I\subseteq P$, we let
$\DF(I)= \langle \DF(f) \mid f\in I \setminus \{0\}\rangle$
be the {\bf degree form ideal} of~$I$.
Passage to the ring $R_\X$ and the induced filtration~$\F$
leads to the following notions.

\begin{definition}
Let $\X$ be a 0-dimensional subscheme of~$\AA^n_K$.
\begin{enumerate}
\item[(a)] For an element $f\in R_\X \setminus \{0\}$ of order
$d=\ord_\F(f)$, the residue class
$\LF(f) = f + F_{d-1} R_\X$ in $\grF(R_\X)$ is called the 
{\bf leading form} of~$f$ with respect to~$\F$.

\item[(b)] The ring $\gr_\F(R_\X)= \bigoplus_{i\in\ZZ} F_i R_\X / 
F_{i-1} R_\X$ is called the {\bf associated graded ring} of~$R_\X$
with respect to~$\F$.
\end{enumerate}
\end{definition}

In our setting the associated graded ring
$\grF(R_\X)$ is a 0-dimensional local ring with maximal
ideal $\langle \bar{x}_1,\dots,\bar{x}_n\rangle$. Its $K$-vector space
dimension is given by
$\dim_K(\grF(R_\X))= \sum_{i=0}^ \infty \dim_K(F_i R_\X / F_{i-1}R_\X)
=\dim_K(R_\X)$. Every non-zero homogeneous element of~$\grF(R_\X)$ is of the
form $\LF(f)$ for some $f\in R_\X \setminus\{0\}$.
For the algorithms in the later sections, the most important
property of the associated graded ring is that it can be computed 
explicitly via the formula $\grF(R_\X) \cong P/\DF(I_\X)$.

\bigbreak
%
%

\section{The Border Basis Scheme}
\label{The Border Basis Scheme}

In the following we let~$K$ be a field,
and we let $\OO=\{t_1,\dots,t_\mu\}$ be an order
ideal in~$\mathbb{T}^n$. The $\OO$-border basis scheme
is a moduli scheme which parametrizes all 0-dimensional ideals having
an $\OO$-border basis.
Let us recall its definition and its defining ideal.

\begin{definition}\label{borderbasisfamily}
Let $\OO=\{t_1,\dots,t_\mu\}$ be an order ideal, and let
$\partial\OO=\{b_1,\dots,b_\nu\}$ be the border of~$\OO$.
\begin{enumerate}
\item[(a)] Let $C=\{c_{ij} \mid i\in\{1,\dots,\mu\}, j\in\{1,\dots,\nu\}\}$ 
be a set of new indeterminates, and let $K[C]=K[c_{11},\dots,c_{\mu\nu}]$.
Then the set of polynomials $G=\{g_1,\dots,g_\nu\}$ in $K[C][x_1,\dots,x_n]$, 
where
$$
g_j \;=\; b_j - c_{1j} t_1 - \cdots - c_{\mu j} t_\mu
$$
for $j=1,\dots,\nu$, is called the {\bf generic $\OO$-border prebasis}.

\item[(b)] For $r=1,\dots,n$, the matrix
$\A_r = (a_{ij}^{(r)}) \in \Mat_\mu(K[C])$, where
$K[C]=K[c_{11},\dots,c_{\mu\nu}]$ and
$$
a_{ij}^{(r)} \;=\; \begin{cases}
\delta_{im} & \hbox{\ \rm if\ } x_r t_j = t_m\\
c_{im} & \hbox{\ \rm if\ } x_r t_j = b_m
\end{cases}
$$
is called the $r$-th {\bf generic multiplication matrix} for~$\OO$.

\item[(c)] Consider the ideal in $K[C]$ which is
generated by all entries of the commutator matrices
$\A_r \A_s - \A_s \A_r$
with $1\le r<s\le n$. Then the subscheme of $\mathbb{A}^{\mu\nu}_K
= \Spec(K[C])$ defined by this ideal is called the
{\bf $\OO$-border basis scheme}. It is denoted by~$\BO$, the 
defining ideal is denoted by~$I(\BO)$, and the corresponding
affine coordinate ring is denoted by
$B_\OO = K[C]/I(\BO)$.

\item[(d)] The ring homomorphism
$$
B_\OO \;\longrightarrow\; U_\mathcal{O} :=
B_\OO[x_1,\dots,x_n] / \langle g_1,\dots,g_\nu\rangle
$$
is called the {\bf universal $\OO$-border basis family}.
\end{enumerate}
\end{definition}

The main reason why $\BO$ is a good moduli space is the following
result (see~\cite{KR4}, Theorem~3.4).

\begin{theorem}
The residue classes of the elements of~$\OO$ are
a $B_\OO$-basis of~$U_\OO$.
\end{theorem}

Thus the universal $\OO$-border basis family
is flat and parametrizes all $\OO$-border bases.
Its $K$-rational points correspond to $\OO$-border bases
in the following way.

\begin{definition}
Let $\Gamma=(\gamma_{ij})\in K^{\mu\nu}$
be a $K$-rational point of~$\BO$. 
Then the polynomials $g_j(x_1,\dots,x_n,\gamma_{11},\dots,
\gamma_{\mu\nu})$ with $j\in\{1,\dots,\nu\}$ form an $\OO$-border
basis. Let $I_\Gamma$ be the ideal in~$P$ which is generated
by these polynomials. 
Then the 0-dimensional scheme $\X_\Gamma$ in~$\AA^n_K$
defined by $I_\Gamma$ is called the 0-dimensional scheme
{\bf represented by~$\Gamma$}.

Conversely, given a 0-dimensional scheme~$\X$ in~$\AA^n_K$
whose vanishing ideal $I_\X$ has an $\OO$-border basis,
the coefficients of that $\OO$-border basis define a
$K$-rational point $\Gamma_\X$ of~$\BO$. We say that the
point $\Gamma_\X$ {\bf represents} the 0-dimensional scheme~$\X$
in~$\BO$.
\end{definition}

Using this terminology, our goal in the next sections is to
describe the subsets of~$\BO$ whose $K$-rational points
represent the 0-dimensional schemes that are locally Gorenstein,
strict Gorenstein schemes, Cayley-Bacharach and strict
Cayley-Bacharach schemes, or strict complete intersections.
Notice that these subsets will frequently be neither the set 
of $K$-rational points of an open
nor of a closed set in the Zariski topology, but merely 
the set of $K$-rational points of 
a {\bf constructible} subset of~$\BO$. As the case may be, we
shall try to describe their structure as explicitly as possible.

An important property of~$I(\BO)$ is that
it is homogeneous with respect to the following grading.

\begin{definition}
Let $\OO=\{t_1,\dots,t_\mu\}$ be an order ideal, let
$\partial\OO=\{b_1,\dots,b_\nu\}$ be the border of~$\OO$, and let
$C=\{c_{ij} \mid i\in\{1,\dots,\mu\},\, j\in\{1,\dots,\nu\}\}$
be the set of indeterminates representing the coefficients of the generic
$\OO$-border prebasis. Then the $\ZZ$-grading on $K[C]$ defined
by  $\deg(c_{ij})=\deg(b_j)-\deg(t_i)$
for $i=1,\dots,\mu$ and $j=1,\dots,\nu$ is called the
{\bf total arrow degree}.
\end{definition}

The name of this grading derives from the fact that we may view
$c_{ij}$ as an {\bf arrow} pointing from~$b_j$ to~$t_i$ as
in~\cite{Hai}, p.~210 and~\cite{Hui2}, Section~3. The next proposition
shows why the total arrow degree is useful for us.

\begin{proposition}
The ideal $I(\BO)$ in~$K[C]$  defining the border basis scheme 
is homogeneous with respect to the total arrow degree.
\end{proposition}

\begin{proof}
First we note that the generic $\OO$-border prebasis~$G$ is
homogeneous with respect to the total arrow degree, if we let
$\deg(x_i)=1$ for $i=1,\dots,n$, as usual.
Hence the universal family $U_\OO= K[C][x_1,\dots,x_n]/\langle G\rangle$
is a graded ring with respect to this grading.
Now we observe that~$\OO$ is a homogeneous $B_\OO$-basis of~$U_\OO$.
For $i=1,\dots,n$, the generic multiplication
matrix~$\A_i$ expresses the multiplication by~$x_i$
in this basis. Thus it yields a homogeneous $B_\OO$-linear map
$\mu_{x_i}:\; \bigoplus_{j=1}^\mu B_\OO(-\deg(t_j)-1) \longrightarrow
\bigoplus_{j=1}^\mu B_\OO(-\deg(t_j))$ of degree zero.
Consequently, the commutators $\A_k \A_\ell -
\A_\ell \A_k$ are homogeneous $B_\OO$-linear
maps of degree zero, and their entries are homogeneous polynomials with
respect to the total arrow degree, as was to be shown.
\end{proof}

In~\cite{Sip}, Prop.~3.2.6, it is shown that~$I(\BO)$ is even homogeneous
with respect to the $\ZZ^n$-grading $\deg(c_{ij})=\log(b_j)-\log(t_i)$
which is called the {\bf arrow grading}. This yields another proof for 
the above proposition. Let us check it in a concrete case.

\begin{example}\label{1-x-y-xy}
Let $K$ be a field, let $P=K[x,y]$, and let $\OO= \{1, x, y, xy\}$. 
Then the border of~$\OO$ is given by
$\partial\OO=\{x^2,\,y^2,\, x^2y,\, xy^2\}$. 
The generic multiplication matrices are
$$
\A_x =
\left( \begin{array}{cccc}
0 & c_{11} & 0 & c_{13} \\
1 & c_{21} & 0 & c_{23} \\
0 & c_{31} & 0 & c_{33} \\
0 & c_{41} & 1 & c_{43} \end{array}\right)
\hbox{\quad and \quad }
\A_y=
\left( \begin{array}{cccc}
0 & 0 & c_{12} & c_{14} \\
0 & 0 & c_{22} & c_{24} \\
1 & 0 & c_{32} & c_{34} \\
0 & 1 & c_{42} & c_{44} \end{array}\right)
$$
Consequently, the defining ideal $I(\B_{\OO})$ 
of the border basis scheme $\BO$ is generated by the entries of
$\A_x \A_y - \A_y \A_x$, i.e., by the polynomials 
$$
\left.\begin{array}{ll}
\{\;
  c_{11}c_{22} +c_{13}c_{42} -c_{14},&
  c_{11}c_{24} -c_{12}c_{33} -c_{14}c_{43} +c_{13}c_{44}, \\
  c_{12}c_{31} +c_{14}c_{41} -c_{13},&
  c_{21}c_{22} +c_{23}c_{42} +c_{12} -c_{24},\\ 
  c_{21}c_{24} -c_{22}c_{33} -c_{24}c_{43} +c_{23}c_{44} +c_{14},\quad &
  c_{22}c_{31} +c_{24}c_{41} -c_{23},\\
  c_{22}c_{31} +c_{33}c_{42} -c_{34},&
  c_{22}c_{41} +c_{42}c_{43} +c_{32} -c_{44},\\
  c_{24}c_{31} -c_{32}c_{33} -c_{34}c_{43} +c_{33}c_{44} -c_{13},&
  c_{31}c_{32} +c_{34}c_{41} +c_{11} -c_{33},\\
  c_{31}c_{42} +c_{41}c_{44} +c_{21} -c_{43},&
  c_{33}c_{42} -c_{24}c_{41} +c_{23} -c_{34}
\; \}
 \end{array}\right.
$$
The degree tuple of the total arrow degree is
$$
(\deg(c_{11}), \deg(c_{12}), \dots, \deg(c_{44})) \;=\; 
(2,2,3,3,\; 1,1,2,2,\; 1,1,2,2,\; 0,0,1,1)
$$
and it is easy to verify that the above polynomials are indeed
homogeneous with respect to the grading it defines.
As shown in~\cite{KR4}, Example~3.8, the scheme~$\BO$
is isomorphic to an 8-dimensional affine space over~$K$.
\end{example}

\bigbreak
%
%

\section{The Locally Gorenstein Locus}
\label{The Locally Gorenstein Locus}

The first subscheme of the border basis scheme which we want
to describe explicitly is the open subscheme parametrizing 0-dimensional
locally Gorenstein schemes. Recall that a local ring~$S$ is said to be
{\bf Gorenstein} if its {\bf socle}, i.e., the annihilator
of its maximal ideal~$\mathfrak{n}$ is a 1-dimensional 
$S/\mathfrak{n}$-vector space.
For equivalent definitions, see for instance~\cite{E}, Ch.~21.

\begin{definition}
Let $\X$ be a 0-dimensional subscheme of~$\AA^n_K$, let
$I_\X=\Q_1\cap \cdots \cap \Q_s$ be the primary decomposition
of its vanishing ideal in~$P$, and let $\q_i$ be the image 
of~$\Q_i$ in~$R_\X=P/I_\X$ for $i=1,\dots,s$.
\begin{enumerate}
\item[(a)] The ring~$R_\X$ is called a {\bf locally Gorenstein ring}
if $R_\X/\q_i$ is a local Gorenstein ring for $i=1,\dots,s$.

\item[(b)] The scheme~$\X$ is said to be {\bf locally Gorenstein}
if~$R_\X$ is a locally Gorenstein ring.
\end{enumerate}
\end{definition}

Sometimes locally Gorenstein schemes are simply called Gorenstein,
if no confusion can arise.
An algorithm for checking the locally Gorenstein property
was given in~\cite{KLR1}, Alg.\ 5.4.

Now we let $\OO=\{t_1,\dots,t_\mu\}$ be an order ideal in~$\mathbb{T}^n$.
Our goal is to describe the following subset 
of the border basis scheme~$\BO$.

\begin{definition}
The set of all $K$-rational points $\Gamma=(\gamma_{ij})
\in K^{\mu\nu}$ of the border basis scheme~$\BO$
which represent a locally Gorenstein 0-dimensional scheme~$\X_\Gamma$ 
is called the {\bf set of locally Gorenstein points} of~$\BO$. 
\end{definition}

Based on Algorithm~5.4 in~\cite{KLR1}, we can describe
the set of locally Gorenstein points of~$\BO$. As the following
algorithm shows, it is the set of $K$-rational points
of an open subscheme  $\LGor_\OO$ of~$\BO$ which we call
the {\bf locally Gorenstein locus} in~$\BO$.

\begin{algorithm}{\upshape\bf (Computing the Locally Gorenstein 
Locus)}\label{alg:GorLoc}\\
Let $\mathcal{O}=\{t_1,\dots,t_\mu\}$ be an order ideal
in~$\mathbb{T}^n$. Consider the following sequence of instructions.
\begin{enumerate}
\item[(1)] Determine the generic multiplication matrices 
$\A_1,\dots,\A_n$ for~$\OO$ in the ring
$\Mat_\mu(K[C])$.

\item[(2)] Calculate the commutators $\A_r \A_s -
\A_s \A_r$ for $1\le r<s\le n$ and form
the ideal $I(\BO)$ in $K[C]$ generated by their entries.

\item[(3)] Introduce new indeterminates $z_1,\dots,z_\mu$
and construct the matrix $D$ in $\Mat_\mu(K[C][z_1,\dots,z_\mu])$
whose $i$-th column is given by
$$
t_i(\A_1\tr,\dots, \A_n\tr) \cdot (z_1,\dots,z_\mu)\tr
$$
for $i=1,\dots,\mu$.

\item[(4)] Compute $\det(D)$ in $K[C][z_1,\dots,z_\mu]$, and
let~$J$ be the ideal in~$K[C]$ generated by the coefficients
of $\det(D)$ with respect to the indeterminates $z_1,\dots,z_\mu$.

\item[(5)] Return the ideal $I(\BO) + J$.
\end{enumerate}
This is an algorithm which computes an ideal in~$K[C]$
which defines a closed subscheme $\NonLGor_\OO$ of~$\BO$
such that the set of $K$-rational points of the complement
$\BO\setminus \NonLGor_\OO$ is precisely the set of
locally Gorenstein points of~$\BO$.
\end{algorithm}

\begin{proof}
Let us apply~\cite{KLR1}, Alg.\ 5.4, where we use
the basis $B=(\bar{t}_1,\dots,\bar{t}_\mu)$ of~$R_\X$.
For a fixed $K$-rational point $\Gamma=(\gamma_{ij})$
of~$\BO$, the corresponding ideal $I_\Gamma$ defines a locally 
Gorenstein scheme if and only if the polynomial obtained by substituting
$(\gamma_{ij})$ for $(c_{ij})$ in~$\det(D)$ is non-zero.
Thus the non-locally Gorenstein locus is defined by equating all
coefficients of $\det(D)$ to zero, and the claim follows.
\end{proof}

In view of this algorithm we see that the locally Gorenstein locus
$\LGor_\OO = \BO \setminus \NonLGor_\OO$ is an open subscheme of~$\BO$.
Let us compute this locus in the setting of Example~\ref{1-x-y-xy}.

\begin{example}\label{1-x-y-xy-cont}
For the order ideal $\mathcal{O}=\{1,x,y,xy\}$, 
let us compute the ideal defining the non-locally
Gorenstein locus in~$\BO$. The ideal $I(\BO)$ was
computed in Example~\ref{1-x-y-xy}.
Let $\mathcal{Z} = (z_1, z_2, z_3, z_4)\tr$, and let~$D$ be the matrix
$D = (\mathcal{Z},\; \mathcal{A}\tr_x\mathcal{Z},\;   
\mathcal{A}\tr_y\mathcal{Z},\; \mathcal{A}\tr_x \mathcal{A}\tr_y\mathcal{Z})$.
Its four columns are 
{\small
$$
\mathcal{Z},\quad
\left(\begin{array}{llll}
0 & 1 & 0 & 0\\
c_{11} & c_{21} & c_{31} & c_{41}\\
0 & 0 & 0 & 1\\
c_{13} & c_{23} & c_{33} & c_{43}\
\end{array}\right)
\mathcal{Z},\quad
\left(\begin{array}{llll}
0 & 0 & 1 & 0\\
0 & 0 & 0 & 1\\
c_{12} & c_{22} & c_{32} & c_{42}\\
c_{14} & c_{24} & c_{34} & c_{44}
\end{array}\right)
\mathcal{Z}, \quad
\left(\begin{array}{llll}
0 & 0 & 0 & 1\\
p_1 & p_2 & p_3 & p_4\\
c_{14} & c_{24} & c_{34} & c_{44}\\
q_1 & q_2 & q_3 & q_4
\end{array}\right)
\mathcal{Z}$$
}%
where $p_1 = c_{12}c_{31} + c_{14}c_{41}$, 
$p_2 = c_{22}c_{31} + c_{24}c_{41}$, 
$p_3 = c_{31}c_{32} + c_{34}c_{41}+c_{11}$, 
$p_4 = c_{31}c_{42}+c_{41}c_{44}+c_{21}$,
$q_1 = c_{12}c_{33}+c_{14}c_{43}$, 
$q_2 = c_{22}c_{33}+c_{24}c_{43}$,
$q_3 = c_{32}c_{33}+c_{34}c_{43}+c_{13}$, 
and $q_4 = c_{33}c_{42}+c_{43}c_{44}+c_{23}$.

The determinant of~$D$ is a polynomial 
\begin{align*}
\det(D) \;=&\;  (-c_{12}^2 c_{13}c_{31} + c_{11}c_{12}^2c_{33} 
-c_{12}c_{13}c_{14}c_{41} +c_{11}c_{12}c_{14}c_{43} -c_{11}c_{14}^2)\, z_1^4
+ \cdots\\ 
&\; \cdots +  (-c_{41}c_{42} +1)\, z_4^4
\end{align*}
in $K[C][z_1,z_2,z_3,z_4]$
which is homogeneous of degree 4 with respect to $z_1,\dots,z_4$
and has 35 non-zero coefficients in~$K[C]$.
Let~$J$ be the ideal generated by these coefficients.
Then the set $\NonLGor(\OO)$ is defined by the ideal
$I(\BO)+J$. 

Here we can compute a Gr\"obner basis of~$I(\BO)+J$
and check that we have $\dim(K[C]/(I(\BO)+J))=4$.
Hence $\NonLGor(\OO)$ is the set of closed points of a
4-dimensional closed subscheme of the 8-dimensional scheme~$\BO$.
\end{example}

\bigbreak
%
%

\section{The Degree Filtered Border Basis Scheme $\BOdf$}
\label{The Degree Filtered Border Basis Scheme BODF}

As before, we let $K$ be a field, $P=K[x_1,\dots,x_n]$,
and $\OO=\{t_1,\dots,t_\mu\}$ an order ideal in~$\mathbb{T}^n$
with border $\partial\OO=\{b_1,\dots,b_\nu\}$.
In this section we focus on a subscheme of
the border basis scheme~$\BO$ which parametrizes
those schemes~$\X$ for which~$\OO$ yields a particularly
nice $K$-basis of~$R_\X$. More precisely, recall that the degree 
filtration $(F_i P)_{i\in\mathbb{Z}}$ on~$P$
is defined by $F_i P = \{ f\in P \mid \deg(f) \le i\} \;\cup\; \{0\}$.
For every $i\in\mathbb{Z}$, let $F_iI=F_iP\cap I$, and let
$F_iR = F_iP/F_iI$. Then the family $(F_iI)_{i\in\mathbb{Z}}$
is called the {\bf induced filtration} on~$I$, and the family
$\mathcal{F}= (F_iR)_{i\in\mathbb{Z}}$
is a $\mathbb{Z}$-filtration on~$R$
which is called the {\bf degree filtration} on~$R$.
For more details on filtrations, see for instance~\cite{KLR1} and 
Section 6.5 of~\cite{KR2}.

\goodbreak
\begin{definition}
Let $\X$ be a 0-dimensional subscheme of~$\AA^n_K$
with vanishing ideal $I_\X \subseteq P$ and affine
coordinate ring $R_\X=P/I_\X$.
\begin{enumerate}
\item[(a)]  A tuple $B=(\bar{t}_1,\dots,\bar{t}_\mu) \in R_\X^\mu$ 
is called a {\bf degree filtered $K$-basis} of~$R_\X$ if
the set $F_iB = B\cap F_i R_\X$ is a $K$-basis of~$F_i R_\X$
for every $i\in\ZZ$ and if $\ord_\F(\bar{t}_1) \le \cdots 
\le \ord_\F(\bar{t}_\mu)$.

\item[(b)] Given an order ideal $\OO=\{t_1,\dots,t_\mu\}$
in~$\mathbb{T}^n$, we say that~$\X$ has a {\bf degree filtered
$\OO$-border basis} if the tuple of residue classes
$(\bar{t}_1,\dots,\bar{t}_\mu)$ is a degree filtered $K$-basis 
of~$R_\X$.
\end{enumerate}
\end{definition}

For the sake of completeness, the following example borrowed from~\cite{KLR1}
provides a monomial $K$-basis which is not degree-filtered.

\begin{example}\label{nondegreefiltered}
Let $K=\mathbb{Q}$, let $P = K[x,y]$,
let $I$ be the vanishing ideal of the affine set
of eight points given by
$p_1=(1,-1)$, $p_2 =(0,2)$, $p_3=(1,1)$,
$p_4 =(1,2)$, $p_5=(0,1)$, $p_6=(1,3)$,
$p_7 = (2,4)$, and $p_8 =(3,4)$, and let $R=P/I$.
The reduced Gr\"{o}bner basis of~$I$ with respect to
\texttt{DegRevLex} is
\begin{gather*}
\{ \; x^{2}y -4x^{2} - xy + 4x,\quad 
x^{3} + xy^{2} -6x^{2} -3xy - y^{2} + 7x + 3y -2,\\
y^{4} -10xy^{2} -5y^{3} + 15x^{2} + 30xy + 15y^{2} -35x -25y + 14,\\
xy^{3} -7xy^{2} - y^{3} + 14xy + 7y^{2} -8x -14y + 8\;\}
\end{gather*}
Since this term ordering is degree compatible,
the residue classes of the elements in the tuple
$(1, y, x, y^2, xy, x^2, y^3, xy^2)$
form a degree-filtered $K$-basis of~$R$
with order tuple $(0,1,1,2,2,2,3,3)$.
On the other hand, the reduced Gr\"{o}bner basis of~$I$ 
with respect to~\texttt{Lex} is
\begin{gather*}
\{\, x^{2} -\tfrac{2}{3}xy^{2} + 2xy -\tfrac{7}{3}x +
\tfrac{1}{15}y^{4} -\tfrac{1}{3}y^{3}
+ y^{2} -\tfrac{5}{3}y + \tfrac{14}{15},\\
xy^{3} -7xy^{2} + 14xy -8x - y^{3} + 7y^{2} -14y + 8,\,
y^{5} -9y^{4} + 25y^{3} -15y^{2} -26y + 24 \,\}
\end{gather*}
So, the residue classes of the elements in the tuple
$B=(1, y, x,  y^2, xy, y^3, xy^2, y^4)$ form a $K$-basis of~$R$.
Since $\bar{y}^4 = 10\bar{x}\bar{y}^2
+ 5\bar{y}^3 - 15\bar{x}^2 - 30\bar{x}\bar{y}
- 15\bar{y}^2 + 35\bar{x} + 25\bar{y} - 14$, 
we have $\ord_\Fbar(\bar{y}^4)=3$.
Altogether, we see that~$B$ is not a degree-filtered basis, since
its order tuple is $(0,1,1,2,2,3,3,3)$.
\end{example}

In addition to these properties, recall that 
$\bar{t}_1=1$ in each degree filtered $K$-basis of~$R_\X$
by Assumption~\ref{orderOO}.
When~$\OO$ is a degree filtered $K$-basis of~$R_\X$, the Hilbert 
function of~$\OO$ agrees with the Castelnuovo function of~$\X$.
For a discussion of these notions, we refer to~\cite{KLR2}, Section~5.

Using this terminology, the degree filtered border basis scheme
is a subscheme of~$\BO$ which
parametrizes all 0-dimensional ideals in~$P$
which have a degree filtered $\OO$-border basis.
The following proposition provides an explicit description
of this subscheme.

\begin{proposition}\label{CharBOdf}
Let $\OO=\{t_1, \dots, t_\mu\}$ be an order ideal, let
$\partial\OO=\{b_1,\dots,b_\nu\}$ be the border of~$\OO$, 
and let $G=\{g_1,\dots,g_\nu\}$
be the generic $\OO$-border prebasis, where
$g_j = b_j -\sum_{i=1}^\mu c_{ij}\, t_i$.
\begin{enumerate}
\item[(a)] For a $K$-rational point $\Gamma=(\gamma_{ij})$ of~$\BO$,
the 0-dimensional scheme $\X_\Gamma$ represented by~$\Gamma$
has a degree filtered $\OO$-border basis if and only if $\gamma_{ij}=0$
for all $i\in\{1,\dots,\mu\}$ and $j\in\{1,\dots,\nu\}$ such that
$\deg(t_i)>\deg(b_j)$.

\item[(b)] Let $I_\OO\df$ be the ideal in~$K[C]$
generated by all indeterminates $c_{ij}$ such that $\deg(t_i)>
\deg(b_j)$. The $K$-rational points of the closed subscheme~$\BOdf$ 
of~$\BO$ defined by the ideal $I(\BO)+I_\OO\df$ 
represent the 0-dimensional schemes $\X$ in~$\AA^n_K$ 
which have a degree filtered $\OO$-border basis.
\end{enumerate}
\end{proposition}

\begin{proof}
Claim~(a) follows immediately from Condition~(d) 
in~\cite{KLR2}, Prop.~5.3, and~(b) is a consequence of~(a).
\end{proof}

As pointed out by the referee, this proposition shows that $\BOdf$ is the 
positive Bia\l ynicki-Birula decomposition of the scheme~$\BO$.
This follows from the general theory of Bia\l ynicki-Birula decompositions for
affine schemes, see for example~\cite{Dri13}, Subsection~1.3.4 and \cite{JS19}, Prop.~4.5.
Thus the scheme $\BOdf$ is linked to the geometry of the Bia\l ynicki-Birula decomposition 
for the Hilbert scheme, as $\BOdf$ is an open subscheme of it (see~\cite{Dri13}, Lemma 1.4.7
and \cite{JS19}, Prop.~5.2).

Part~(b) of the preceding proposition gives rise to the following definition.

\begin{definition}
Let $\OO=\{t_1,\dots,t_\mu\}$ be an order ideal in~$\mathbb{T}^n$,
let $\partial\OO=\{b_1,\dots,b_\nu\}$ be the border of~$\OO$, and
let $I_\OO\df$ be the ideal in~$K[C]$ generated by all 
indeterminates $c_{ij}$ such that $\deg(t_i)>\deg(b_j)$.
\begin{enumerate}
\item[(a)] The closed subscheme~$\BOdf$ of~$\BO$ defined by 
$I(\BOdf) = I(\BO)+I_\OO\df$ is called the
{\bf degree filtered $\OO$-border basis scheme}.
Its affine coordinate ring is denoted by $B_\OO\df =
K[C]/I(\BOdf)$.

\item[(b)] The set of polynomials $G\df = \{g_1\df,\dots,
g_\nu\df\}$ in~$K[C][x_1,\dots,x_n]$ given by 
$g_j = b_j - \sum_{\{i\mid \deg(t_i)\le\deg(b_j)\}}
c_{ij}\, t_i$  for $j=1,\dots,\nu$ 
is called the {\bf generic degree
filtered $\OO$-border prebasis}.

\item[(c)] The canonical $B_\OO\df$-algebra homomorphism
$$
\Phi\df:\; B_\OO\df \longrightarrow U_\OO\df :=
B_\OO\df[x_1,\dots,x_n]/ \langle G\df\rangle
$$ 
is called the {\bf universal degree filtered $\OO$-border basis family}.
\end{enumerate}
\end{definition}

Since the universal $\OO$-border basis family is a free $B_\OO$-module
with basis~$\OO$ (cf.~\cite{KR4}, Thm.~3.4), it follows by a base change that
the universal degree filtered $\OO$-border basis family is free with basis~$\OO$,
too. The matrices defining the multiplication maps on~$U_\OO\df$
can be obtained as follows.

\begin{remark}\label{MultMatUdf}
For $k=1,\dots,n$, let $\A_k \in \Mat_\mu(K[C])$ be
the $k$-th generic multiplication matrix with respect to~$\OO$
(see Definition~\ref{borderbasisfamily}.b), 
and let $C^{\rm nondf}$ be the set of all
indeterminates $c_{ij}$ such that $\deg(t_i)>\deg(b_j)$, i.e.,
the set of all indeterminates $c_{ij}$ of negative total arrow degree.
\begin{enumerate}
\item[(a)] For $k=1,\dots,n$, let $\A_k\df$
be the matrix obtained from $\A_k$ by setting all
indeterminates in~$C^{\rm nondf}$ equal to zero. Then the matrix
$\A_k\df$
describes the multiplication by~$x_k$ on~$U_\OO\df$
with respect to the basis~$\OO$. The matrices
$\A_1\df, \dots, \A_n\df$ are called the
{\bf generic degree filtered multiplication matrices}
with respect to~$\OO$.

\item[(b)] For any polynomial $f\in P$, the multiplication by~$f$
on~$U_\OO\df$ is given by the matrix $f(\A_1\df,
\dots,\A_n\df)$ with respect to the basis~$\OO$.
\end{enumerate}
\end{remark}

Another useful observation is that the total arrow degree
yields a non-negative grading of~$B_\OO\df$ in the following
sense.

\begin{remark}\label{arrownonneg}
Given the set $C=\{c_{ij}\mid i\in\{1,\dots,\mu\},\; 
j\in\{1,\dots,\nu\}\}$, we let $C\df = \{c_{ij}\in C \mid \deg(t_i)\le \deg(b_j)\}$ 
and $C^{\rm nondf} = C\setminus C\df$. When we set the indeterminates
in~$C^{\rm nondf}$ equal to zero in~$I(\BO)$, we get an ideal
$\bar{I}(\BOdf)$ such that 
$$
B_\OO\df \;=\; K[C]/I(\BOdf) \;\cong\; K[C\df]/\bar{I}(\BOdf) \eqno{(\ast)}
$$
Notice that the indeterminates in~$C\df$ are precisely the
elements of non-negative total arrow degree in~$C$. Hence
the isomorphism $(\ast)$ shows that the total arrow degree
provides a non-negative grading of~$B_\OO\df$.
Another way of phrasing this observation is that the
entries of the commutators of the matrices~$\A_i\df$
are homogeneous polynomials of non-negative degrees with respect
to the total arrow degree.
\end{remark}

The next proposition provides an important family of examples 
of degree filtered border bases. To prove it, we use the
following auxiliary result.

\begin{lemma}
Let $\OO=\{t_1,\dots,t_\mu\}$ be an order ideal
with border $\partial\OO=\{b_1,\dots,b_\nu\}$.
Recall that $\deg(t_1)\le \cdots \le \deg(t_\mu)$, and w.l.o.g.\ let
$\deg(b_1)\le \cdots \le\deg(b_\nu)$. 

Assume that $\deg(t_\mu)>\deg(b_1)$, and let $\Gamma=(\gamma_{ij})\in K^{\mu\nu}$
with $\gamma_{\mu 1}=1$ and $\gamma_{ij}=0$ for $(i,j)\ne (\mu,1)$.
Then~$\Gamma$ represents a 0-dimensional scheme $\X_\Gamma$ in~$\BO$.
\end{lemma}

\begin{proof} To show that the $\OO$-border prebasis
$G_\Gamma = (g_1,\dots,g_\nu) = (b_1-t_\mu, b_2,\dots, b_\nu)$ 
is an $\OO$-border basis,
it suffices to show that the residue classes of the 
elements of~$\OO$ form a $K$-basis of $P/I_\Gamma$, 
where $I_\Gamma=\langle G_\Gamma \rangle$. 

Since~$\OO$ is an order ideal,
we have $b_1 \nmid t_\mu$. Therefore there exists a term ordering~$\sigma$
such that $b_1 >_\sigma t_\mu$. Now we use Buchberger's Criterion
to verify that $G_\Gamma = (b_1-t_\mu, b_2,\dots,b_\nu)$ is a 
$\sigma$-Gr\"obner basis of $I_\Gamma$. 

For this it suffices to show that, for $k=2,\dots,\nu$, the 
$S$-polynomial~$S_{1k}$ satisfies $S_{1k} \stackrel{G}{\rightarrow} 0$.
Using
$$
S_{1k} = \tfrac{{\rm lcm}(b_1,b_k)}{b_1}(b_1-t_\mu) 
-  \tfrac{{\rm lcm}(b_1,b_k)}{b_k} b_k
= -\tfrac{{\rm lcm}(b_1,b_k)}{b_1}t_\mu
$$
and the fact that $\deg(b_1)\le \deg(b_k)$ implies $b_k \nmid b_1$,
we see that $S_{1k}$ is a proper multiple of~$t_\mu$.
Therefore the claim follows from $x_it_\mu\in \partial\OO\setminus\{b_1\}
\subset G_\Gamma$. 
Now we conclude from Macaulay's Basis Theorem 
that $\OO=\OO_\sigma(I_\Gamma)=\mathbb{T}^n \setminus \LT_\sigma(G_\Gamma)$
represents a $K$-basis of $P/I_\Gamma$, and the proof is complete.
\end{proof}

At this point we are ready for the following result.

\begin{proposition}\label{CharMaxdeg}
For an order ideal $\OO=\{t_1,\dots,t_\mu\}$,
the following conditions are equivalent.
\begin{enumerate}
\item[(a)] We have $\BO=\BOdf$.

\item[(b)] For $i=1,\dots,\mu$ and $j=1,\dots,\nu$,
we have $\deg(t_i) \le \deg(b_j)$.

\item[(c)] The order ideal~$\OO$ has the {\bf generic
Hilbert function}, i.e., for $i\ge 0$ we have
$\#\{t_j\in\OO \mid \deg(t_j)\le i\} = 
\min \{ \mu, \# \mathbb{T}^n_{\le i} \}$.

\end{enumerate}
If these conditions are satisfied, we say that~$\OO$
has a {\bf maxdeg border}.
\end{proposition}

\begin{proof}
First we show that~(a) implies~(b). By hypothesis, 
every scheme $\X_\Gamma$ represented by a $K$-rational point~$\Gamma$
of~$\BO$ has an $\OO$-border bases which is degree filtered. Suppose 
that there exist terms $t_i\in\OO$ and $b_j\in\partial\OO$
such that $\deg(t_i)>\deg(b_j)$. Then the lemma yields
a 0-dimensional scheme $\X_\Gamma$ which has an $\OO$-border
basis that is not degree filtered, in contradiction
to the hypothesis.

Since (b)$\Leftrightarrow$(c) is clear, it suffices to note
that (b)$\Rightarrow$(a) follows from~\cite{KLR2}, Prop.~5.3
in order to finish the proof.
\end{proof}

Degree filtered border bases also arise naturally
from degree compatible term orderings as follows.

\begin{example}
Let $\X$ be a 0-dimensional scheme in~$\AA^n_K$, and let~$\sigma$
be a degree compatible term ordering. Then~$\X$ has a degree filtered
border basis with respect to $\mathcal{O}_\sigma(I_\X)=
\mathbb{T}^n \setminus \LT_\sigma(I_\X)$.
This follows from the observation that a border basis element
$g_j= b_j - \sum_{i=1}^\mu \gamma_{ij}t_i$ can be interpreted as
saying that the normal form of~$b_j$ with respect to a $\sigma$-Gr\"obner 
basis of~$I_\X$ is $\sum_{i=1}^\mu \gamma_{ij}t_i$, and since~$\sigma$
is degree compatible, this implies $\deg(\sum_{i=1}^\mu \gamma_{ij}t_i)
\le \deg(b_j)$. Hence we have $b_j\in\DF(g_j)$,
Condition~(d) of~\cite{KLR2}, Prop.~5.3 is satisfied, 
and the border basis is degree filtered.
\end{example}

A particular case of a degree filtered border basis is a homogeneous one.
The family of all homogeneous border bases can be parametrized 
as follows (see~\cite{KR4}, Section~5).

\begin{definition}
Let $\OO=\{t_1,\dots,t_\mu\}$ be an order ideal,
let $\partial\OO=\{b_1,\dots,b_\nu\}$ be the border of~$\OO$, and
let $I_\OO\hom$ be the ideal in~$K[C]$ generated by all 
indeterminates $c_{ij}$ such that $\deg(t_i)\ne\deg(b_j)$.
\begin{enumerate}
\item[(a)] The closed subscheme~$\BOhom$ of~$\BO$ defined by 
$I(\BOhom) = I(\BO)+I_\OO\hom$ is called the
{\bf homogeneous $\OO$-border basis scheme}.
Its affine coordinate ring is denoted by $B_\OO\hom =
K[C]/I(\BOhom)$.

\item[(b)] The set of polynomials $G\hom = \{g_1\hom,\dots,
g_\nu\hom\}$ in~$K[C][x_1,\dots,x_n]$ given by 
$g_j = b_j - \sum_{\{i\mid \deg(t_i)=\deg(b_j)\}}
c_{ij}\, t_i$ for $j=1,\dots,\nu$  is called the 
{\bf generic homogeneous $\OO$-border prebasis}.

\item[(c)] The canonical $B_\OO\hom$-algebra homomorphism
$$
\Phi\hom:\; B_\OO\hom \longrightarrow U_\OO\hom :=
B_\OO\hom[x_1,\dots,x_n]/ \langle G\hom\rangle
$$ 
is called the {\bf universal homogeneous $\OO$-border basis family}.
\end{enumerate}
\end{definition}

Again it follows by a base change from~\cite{KR4}, Thm.~3.4, that
the universal homogeneous $\OO$-border basis family is free with basis~$\OO$.
The matrices defining the multiplication maps on~$U_\OO\hom$
can be described as follows.

\begin{remark}\label{MultMatUhom}
For $k=1,\dots,n$, let $\A_k \in \Mat_\mu(K[C])$ be
the $k$-th generic multiplication matrix with respect to~$\OO$
(see Definition~\ref{borderbasisfamily}.b), 
and let $C^{\rm nonhom}$ be the set of all
indeterminates $c_{ij}$ such that $\deg(t_i)\ne\deg(b_j)$.
\begin{enumerate}
\item[(a)] For $k=1,\dots,n$, let $\A_k\hom$
be the matrix obtained from $\A_k$ by setting all
indeterminates in~$C^{\rm nonhom}$ equal to zero. Then the matrix
$\A_k\hom$ represents the multiplication map by~$x_k$ 
on~$U_\OO\hom$ with respect to the basis~$\OO$. The matrices
$\A_1\hom, \dots, \A_n\hom$ are called the
{\bf generic homogeneous multiplication matrices}
with respect to~$\OO$.

\item[(b)] For any polynomial $f\in P$, the multiplication by~$f$
on~$U_\OO\hom$ is given by the matrix $f(\A_1\hom,
\dots,\A_n\hom)$ with respect to the basis~$\OO$.
\end{enumerate}
\end{remark}

It is clear that the homogeneous $\OO$-border basis scheme $\BOhom$
is a closed subscheme of~$\BOdf$. 
Let us also look at the total arrow degree on~$B_\OO\hom$.

\begin{remark}
Let $C=\{c_{ij} \mid i\in\{1,\dots,\mu\},\, j\in\{1,\dots,\nu\}\}$, 
let $C\df = \{c_{ij}\in C \mid\deg(t_i)\le \deg(b_j)\}$, and let 
$C^{\rm hom} = \{c_{ij} \in C \mid \deg(t_i)=\deg(b_j)\}$.
Then the indeterminates in $C^{\rm hom}$ are precisely the elements
of total arrow degree zero in~$C$. Hence the entries of the
commutators of the matrices $\A_i\hom$ are
homogeneous polynomials of degree zero with respect to 
the total arrow degree. Since they generate~$I(\BO\hom)$, we
obtain an isomorphism of graded $K$-algebras
$$
B_\OO\hom \;=\; K[C]/I(\BOhom) \;\cong\; K[C\hom] / \bar{I}(\BOhom)
$$
where $\bar{I}(\BOhom)$ is the result of setting all indeterminates
in $C^{\rm nonhom} = C\setminus C\hom$ equal to zero in~$I(\BO)$ (or in $I(\BOdf)$),
and where the total arrow degree induces the trivial grading on~$B_\OO\hom$.
\end{remark}

By combining this remark with Remark~\ref{arrownonneg}, we get
the following useful result.

\begin{proposition}\label{IhomInIdf}
Given the set $C=\{c_{ij} \mid i\in\{1,\dots,\mu\},\,j\in\{1,\dots,\nu\}\}$,
we let $C\df = \{c_{ij}\in C \mid \deg(t_i)\le \deg(b_j)\}$ and 
$C\hom = \{c_{ij} \in C \mid \deg(t_i)=\deg(b_j)\}$.
\begin{enumerate}
\item[(a)] The defining ideal $\bar{I}(\BOhom)$
of~$\BOhom$ in~$K[C\hom]$
is exactly the set of homogeneous elements of degree
zero of the defining ideal $\bar{I}(\BOdf)$
in~$K[C\df]$ with respect to the total arrow degree.
In particular, we have $\bar{I}(\BOhom)=
\bar{I}(\BOdf) \cap K[C\hom]$ and
$\bar{I}(\BOhom) \cdot K[C\df] \subseteq \bar{I}(\BOdf)$.

\item[(b)] We have a canonical injective $K$-algebra
homomorphism $B_\OO\hom \hookrightarrow B_\OO\df$,
and the elements of $B_\OO\hom$ are precisely
the elements of total arrow degree zero in~$B_\OO\df$.
\end{enumerate}
\end{proposition}

In more geometric jargon, this proposition shows that $\BOhom$ is
the $\mathbb{G}_m$-fixed locus of~$\BO$, and hence an open affine subscheme
of the $\mathbb{G}_m$-fixed locus of the Hilbert scheme.
Thus the closed immersion $\BOhom \hookrightarrow \BOdf$ can be
interpreted geometrically as the retraction of the Bia\l ynicki-Birula
decomposition to the $\mathbb{G}_m$-fixed locus.

To finish the section, let us compute the subschemes
$\BOdf$ and $\BOhom$ of~$\BO$ in the setting of
Example~\ref{1-x-y-xy}.

\goodbreak
\begin{example}\label{1-x-y-xy-cont2}
For the order ideal $\OO = \{1,x,y,xy\}$ in~$\mathbb{T}^2$, 
we have calculated the ideal $I(\BO)$ in Example~\ref{1-x-y-xy}.
Let us also determine $I(\BOdf)$ and $I(\BOhom)$ in this case.
\begin{enumerate}
\item[(a)] To get the defining ideal of~$\BOdf$,
we find $C^{\rm nondf} = \{c_{ij} \mid \deg(b_j)
<\deg(t_i)\}$ first. In the current example we get $C^{\rm nondf}=
\emptyset$, and therefore $I(\BOdf)=I(\BO)$. In other words, we have
$\BOdf=\BO$.

\item[(b)] The defining ideal of~$\BOhom$ is obtained by adding
the ideal generated by
\begin{align*}
C^{\rm nonhom} &\;=\; \{ c_{ij} \mid \deg(b_j)\ne \deg(t_i) \} \\
&\;=\; \{ c_{11}, c_{12}, c_{13}, c_{14}, c_{21}, c_{22}, c_{23},
c_{24}, c_{31}, c_{32}, c_{33}, c_{34}, c_{43}, c_{44} \}
\end{align*}
to $I(\BO)$, and after simplifying the generators of $I(\BO)$, we obtain
$I(\BOhom) = \langle C^{\rm nonhom} \rangle$. In other words,
the scheme $\BOhom$ is equal to the plane $\Spec(K[c_{41},c_{42}])$.
\end{enumerate}
\end{example}

\bigbreak
%
%

\section{The Cayley-Bacharach Locus in $\BOdf$}
\label{The Cayley-Bacharach Locus in BODF}

In the following we want to describe the locus of all $K$-rational points
of the border basis scheme which represent 0-dimensional
affine schemes~$\X$ having the Cayley-Bacharach property. Since the description of
this property requires us to fix the Hilbert function of~$\X$,
we have to work in an appropriate subscheme of~$\BO$. When the Hilbert function
of~$\X$ corresponds to the degrees of the terms in~$\OO$, the appropriate subscheme
of~$\BO$ is the degree filtered border basis scheme. Hence we look for a 
description of the Cayley-Bacharach locus in~$\BOdf$. Let us start by recalling
the definition of the Cayley-Bacharach property in the general setting
(see~\cite{KLR1}, Definition 3.10).

Let $\X$ be a 0-dimensional subscheme of~$\AA^n_K$, let $I_\X\subseteq P$ be the
vanishing ideal of~$\X$, and let $R_\X = P/I_\X$ be the coordinate ring of~$\X$.
Then the primary decomposition of the vanishing 
ideal~$I_\X$ of~$\X$ has the form
$$
I_\X \;=\; \Q_1 \cap \cdots \cap \Q_s
$$
where $\Q_i\subset P$ is an $\M_i$-primary ideal with
a maximal ideal $\M_i \subset P$ for each $i\in\{1,\dots,s\}$.
The image of the maximal ideal~$\M_i$ in~$R_\X$ is denoted
by $\m_i$ for $i\in\{1,\dots,s\}$. Clearly, the ideals $\m_1,\dots,\m_s$
are the maximal ideals of~$R_\X$.
In this setting, the following definitions were introduced
in~\cite{KLR1}.

\begin{definition} Let $\X$ be a 0-dimensional subscheme
of $\AA^n_K$ as above.
\begin{enumerate}
\item[(a)] For $i\in\{1,\dots,s\}$, an ideal $J \subset P$ 
is called a {\bf minimal $\Q_i$-divisor} of~$I_\X$ if
$J=\Q_1 \cap \cdots \cap \Q'_i \cap \cdots\cap \Q_s$
with an ideal $\Q'_i \subset P$ such that
$\Q_i \subset \Q'_i \subseteq \M_i$ and
$\dim_K(\Q'_i/\Q_i)=\dim_K(P/\M_i)$.

\item[(b)] For $i\in \{1,\dots,s\}$ and a minimal
$\Q_i$-divisor $J$ of~$I_\X$, we let
$\ri(J/I_\X) = \max\{\ord_\F(f) \mid f\in J/I_\X
\setminus \{0\}\}$. Then the number
$$
\sepdeg(\m_i) \;=\; \min\{ \ri(J/I_\X) \mid J
\hbox{\ \rm is a minimal $\Q_i$-divisor of\ }I_\X\}
$$
is called the {\bf separator degree} of~$\m_i$ in~$R_\X$.

\item[(c)] The scheme~$\X$ is called a {\bf Cayley-Bacharach scheme},
or the ring $R_\X$ is said to have the {\bf Cayley-Bacharach property},
if we have $\sepdeg(\m_i)=\ri(R_\X)$ for $i=1,\dots,s$.
\end{enumerate}
\end{definition}

In~\cite{KLR1}, Section~3, it is shown that this definition
generalizes the classical definition of the Cayley-Bacharach property
for a set of points~$\X$ in an affine or projective space over an
algebraically closed field. Recall that this classical definition
is usually phrased by requiring that every hypersurface of degree
$\ri(R_\X)-1$ which contains all points of~$\X$ but one automatically
contains the last point.

Furthermore, in~\cite{KLR1}, Alg.~4.6, we provided an algorithm
which checks the Cayley-Bacharach property of~$R_\X$ using
the multiplication matrices of the canonical module of~$R_\X$. This result
will be the basis of our algorithm below.
First we introduce the following terminology.

\begin{definition}\label{def:CBLocus}
Let $\OO=\{t_1,\dots,t_\mu\}$ be an order ideal in~$\mathbb{T}^n$.
Then the set of all $K$-rational points $\Gamma=(\gamma_{ij})
\in K^{\mu\nu}$ of the border basis scheme~$\BO$
which represent a 0-dimensional Cayley-Bacharach scheme~$\X_\Gamma$
is called the {\bf set of Cayley-Bacharach points} of~$\BO$.
\end{definition}

As the following algorithms shows, there exists an open 
subscheme $\CB^{\rm df}_\OO$ of~$\BOdf$
whose $K$-rational points are precisely the set of 
Cayley-Bacharach points of~$\BOdf$. It is 
called the {\bf Cayley-Bacharach locus} in~$\BOdf$.
More precisely, we compute the equations
of a closed subscheme of~$\BOdf$ which forms the complement 
of~$\CB^{\rm df}_\OO$.

\begin{algorithm}{\upshape\bf (Computing the Cayley-Bacharach Locus 
in $\BOdf$)}\label{alg:CBLocDF}\\
Let $\OO=\{t_1,\dots,t_\mu\}$ be an order ideal in~$\mathbb{T}^n$,
and let $\Delta = \# \{i\in\{1,\dots,\mu\} \mid \deg(t_i)=\deg(t_\mu)\}$.
Consider the following sequence of instructions.
\begin{enumerate}
\item[(1)] Using Proposition~\ref{CharBOdf}.b, calculate
$I(\BOdf)=I(\BO)+I_\OO\df$.

\item[(2)] Construct the generic degree filtered multiplication matrices 
$\A^{\rm df}_1, \dots, \A^{\rm df}_n$. For $i=1,\dots,\mu$, 
compute the multiplication matrix 
$M_{t_i}=t_i(\A^{\rm df}_1,\dots,\A^{\rm df}_n)$.

\item[(3)]
For  $j=1,\dots,\Delta$, form the matrix 
$V_j \in \Mat_\mu(K[C])$  whose $i$-th column is 
the $(\mu-\Delta+j)$-th column of $M_{t_i}\tr$ for $i=1,\dots,\mu$.

\item[(4)] Form the block column matrix
$W=\Col(V_1, \dots, V_\Delta)$ and 
compute the ideal~$J_\OO$ in~$K[C]$ generated by the maximal 
minors of~$W$.

\item[(5)]  Return the ideal $I(\BOdf) + J_\OO$.
\end{enumerate}
This is an algorithm which computes an ideal in~$K[C]$. 
This ideal defines a closed subscheme
$\NonCB^{\rm df}_{\OO}$ whose $K\!$-rational points 
represent those 0-dimensional subschemes
of~$\AA^n_K$ which have a degree filtered $\OO$-border basis,
but are not Cayley-Bacharach schemes.
\end{algorithm}

\begin{proof}
A $K$-rational rational point $\Gamma=(\gamma_{ij})$ of~$\BOdf$
corresponds to a zero of the ideal $I(\BOdf)$ which is computed
in Step~(1). According to Algorithm~4.6 of~\cite{KLR1}, the 0-dimensional 
scheme $\X_\Gamma$ represented by~$\Gamma$ is a Cayley-Bacharach scheme
if and only if~$\Gamma$ is a zero of the ideal generated by the
maximal minors of~$W$. This observation finishes the proof.
\end{proof}

In view of this algorithm, the Cayley-Bacharach locus in~$\BOdf$
is the open subscheme $\CB^{\rm df}_\OO = \BOdf \setminus
\NonCB^{\rm df}_\OO$ of the degree filtered $\OO$-border basis scheme.
Let us compute it for the following order ideal.

\begin{example}\label{1-x-y-z-x^2}
Let $K$ be a field, and let $\OO$ be the order ideal 
$\OO=\{1,x,y,z,x^2\}$ in $P=K[x,y,z]$.
Then we have $\partial\OO = \{ xy,y^2,xz,yz,z^2,x^3,x^2y,x^2z\}$
and thus $\mu=5$ as well as $\nu=8$.
The ring $K[C]$ has 40 indeterminates, and the ideal
$I(\BO)$ has 60 generators. Since~$\OO$ has the generic Hilbert
function, we have $\BO=\BOdf$ by Proposition~\ref{CharMaxdeg}.

When we calculate an ideal defining $\NonCB^{\rm df}_\OO$
using the algorithm, we have $\Delta=1$ and
$$
W \;=\; V_1 \;=\; \begin{pmatrix}
0 & 0 & 0 & 0 & 1\\
0 & 1 & c_{51} & c_{53} & c_{56} \\
0 & c_{51} & c_{52} & c_{54} & \;c_{31}c_{51}+c_{41}c_{53} + c_{51}c_{56}
+ c_{21} \\
0 & c_{53} & c_{54} & c_{55} & \;c_{33}c_{51} + c_{43}c_{53} + 
c_{53}c_{56} + c_{23}\\
\;1\; & \;c_{56}\; & \;c_{57}\; & \;c_{58} & c_{36}c_{51} + c_{46}c_{53} + 
c_{56}^2 + c_{26}
\end{pmatrix}
$$
Consequently, the algorithm returns the ideal 
$I(\BOdf) + \langle \det(W)\rangle \, = \; 
I(\BOdf) + \,\allowbreak \langle f\rangle$, where $f=
c_{52}\, c_{53}^2 - 2\, c_{51}\,c_{53}\,c_{54} + c_{51}^2\, c_{55} 
+ c_{54}^2  - c_{52}\, c_{55}$. Thus the Cayley-Bacharach
locus in~$\BOdf$ is the complement of the hypersurface section
of~$\BOdf$ given by $f=0$.
\end{example}

\bigbreak
%
%

\section{The Strict Cayley-Bacharach and Strict Gorenstein Loci in $\BOdf$}
\label{The Strict Cayley-Bacharach and Strict Gorenstein Loci in BODF}

Stronger properties than the Cayley-Bacharach
property are defined as follows.

\begin{definition}
Let $\X$ be a 0-dimensional scheme in~$\AA^n_K$.
\begin{enumerate}
\item[(a)] The scheme~$\X$ is called
a {\bf strict Cayley-Bacharach scheme} if the
graded ring $\grF(R_\X)$ has the Cayley-Bacharach property.

\item[(b)] The scheme~$\X$ is called a {\bf strict Gorenstein scheme}
if the graded ring $\grF(R_\X)$ is a Gorenstein ring.

\end{enumerate}
\end{definition}

In~\cite{KLR1}, Section~6, we showed that every strict
Cayley-Bacharach scheme is a Cayley-Bacharach scheme.

\subsection{Checking the Strict Cayley-Bacharach Property}

In order to check schemes for the strict Cayley-Bacharach property, 
we may use the following algorithm
which follows from~\cite{KLR1}, Alg.~4.6.

\begin{algorithm}{\upshape\bf (Checking Strict Cayley-Bacharach 
Schemes)}\label{alg:CheckSCBS}\\
Let $\X$ be a 0-dimensional scheme in~$\AA^n_K$ with affine coordinate
ring $R_\X = P/I_\X$, and let $\mu=\dim_K(R_\X)$.
Consider the following sequence of instructions.
\begin{enumerate}
\item[(1)] Compute an order ideal $\OO=\{t_1,\dots,t_\mu\}$
in~$\mathbb{T}^n$ such that~$\OO$ is a degree filtered $K$-basis
of~$R_\X$. Let $\Delta\ge 1$ be such that 
$t_{\mu-\Delta+1},\dots, t_\mu$ are the elements of~$\OO$ 
of order $\ri(R_\X)$.

\item[(2)] For $i=1,\dots,\mu$, compute the matrix
$M_{\bar{t}_i}\in\Mat_\mu(K)$ representing the multiplication 
by~$\bar{t}_i$ on $\grF(R_\X)\cong R_\X/\DF(I_\X)$ with respect to
the $K$-basis $\overline{\OO}=(\bar{t}_1,\dots,\bar{t}_\mu)$.

\item[(3)] For $j=1,\dots,\Delta$, form the 
matrix $V_j \in \Mat_\mu(K)$ whose $i$-th column is 
the $(\mu-\Delta+j)$-th column of $(M_{\bar{t}_i})\tr$ for 
$i=1,\dots,\mu$.

\item[(4)] Form the block column matrix
$W=\Col(V_1,\dots,V_\Delta)$ and compute the maximal minors
of~$W$.

\item[(5)] If one of these maximal minors is non-zero,
return ${\tt TRUE}$. Otherwise, return ${\tt FALSE}$.
\end{enumerate}
This is an algorithm which checks whether~$\X$ is a
strict Cayley-Bacharach scheme and returns the corresponding
Boolean value.
\end{algorithm}

For an example of the application of this algorithm, 
see~\cite{KLR1}, Ex.~6.11. 
\smallskip

A strict Gorenstein scheme is always locally Gorenstein (see for
instance~\cite{KLR1}, Thm.\ 6.8), but the converse is not true in general.
Moreover, a strict Gorenstein scheme is always a strict
Cayley-Bacharach scheme (see~\cite{KLR1}, Thm.\ 6.12).
The Hilbert function of the graded ring $\grF(R_\X)$ is given
by the Castelnuovo function of~$\X$, i.e., by the first difference function
of~$\HFa_\X$. The following property is essential here.

\begin{definition}
For an order ideal $\OO=\{t_1,\dots,t_\mu\}\subseteq\mathbb{T}^n\!$, 
let $\HF_\OO=(h_0,h_1,\dots)$ be the affine Hilbert function of~$\OO$,
and let $\rho=\max \{ i\ge 0 \mid h_i>0\}$.
We say that $\HF_\OO$ is {\bf symmetric} if
$h_i=h_{\rho-i}$ for $i=1,\dots,\rho$.
\end{definition}

Notice that if~$\OO$ is a degree filtered $K$-basis of~$R_\X$
then this definition agrees with the definition of the symmetry
of the Hilbert function of~$\X$ given in~\cite{KLR1}, Def.~6.4.
It is known that the affine Hilbert function of a strict Gorenstein 
scheme is symmetric (see for instance \cite{KLR1}, Thm.~6.8). More
precisely, a 0-dimensional affine scheme~$\X$ is a
strict Gorenstein scheme if and only if $\HFa_\X$ is symmetric and~$\X$ 
is a strict Cayley-Bacharach scheme (see~\cite{KLR1}, Thms.\ 6.8 and 6.12).
Thus, using Algorithm~\ref{alg:CheckSCBS}, we can check 
whether~$\X$ is strictly Gorenstein in the following way.

\begin{corollary}{\upshape\bf (Checking Strict Gorenstein 
Schemes)}\label{alg:CheckSGor}\\
In the setting of Algorithm~\ref{alg:CheckSCBS}, the 
following instructions define an algorithm which
checks whether~$\X$ is a strict Gorenstein scheme 
and returns the corresponding Boolean value.
\begin{enumerate}
\item[(1)] Compute an order ideal $\OO=\{t_1,\dots,t_\mu\}$
which represents a degree filtered $K\!$-basis of~$R_\X$.

\item[(2)] Check whether $\HF_\OO$ is symmetric. If this 
is not the case, return ${\tt FALSE}$ and stop.

\item[(3)] Using Algorithm~\ref{alg:CheckSCBS}, check 
whether~$\X$ is a strict Cayley-Bacharach scheme and return
the corresponding Boolean value.
\end{enumerate}
\end{corollary}

\begin{remark}
Notice that, in the last part of this algorithm, it would have been
sufficient to check whether~$\X$ is a Cayley-Bacharach scheme.
However, the more stringent condition of Algorithm~\ref{alg:CheckSCBS}
yields, in general, a more efficient test. Moreover, notice that
the last steps in Algorithm~\ref{alg:CheckSCBS} simplify here, 
because we have $\Delta=1$ and therefore only one matrix~$V_1$
and one maximal minor $\det(V_1)$.
\end{remark}

Let us apply the preceding algorithm to some concrete cases.

\begin{example}\label{ex:StrGorNotStrCI}
Let $K=\QQ$, let $P=K[x,y,z]$, and let~$\X$ be the 0-dimensional
subscheme of~$\AA^n_K$ defined by the ideal
$I_\X=\langle y^2-x^2,\,z^2-x^2,\, xy,\,xz,\,yz,\,x^3 \rangle$.
Since the ideal~$I_\X$ is homogeneous, we have 
$\DF(I_\X)=I_\X$. Moreover, from the fact that the given
system of generators is the reduced Gr\"obner
basis of~$I_\X$ with respect to the ${\tt DegRevLex}$ term ordering
such that $z>y>x$,
it follows that a degree filtered $K$-basis of~$R_\X$
is given by $\OO=\{ 1,\, x,\, y,\, z,\, x^2\}$.

Thus the Hilbert function of~$\OO$ is $(1,3,1)$ which is
symmetric. Hence Algorithm~\ref{alg:CheckSGor} asks us to
compute the multiplication matrices for the elements of~$\OO$
and to combine their last columns. We get the matrix
$$
V \;=\; \begin{pmatrix}
\;0\; & \;0\; & \;0\; & \;0\; & \;1\;  \\
0 & 1 & 0 & 0 & 0\\
0 & 0 & 1 & 0 & 0\\
0 & 0 & 0 & 1 & 0\\
1 & 0 & 0 & 0 & 0\\
\end{pmatrix}
$$
and since $\det(V)=-1$, we conclude that~$\X$ is a strict
Gorenstein scheme.
\end{example}

The following example of a strict Gorenstein scheme
is the well-known case of eight points on a twisted cubic curve.
Notice that in~\cite{KLR2}, Example~4.6, we check that 
it is not a strict complete intersection.

\begin{example}\label{ex:8pointscubic}
Let $K=\QQ$, let $P=K[x,y,z]$, and let~$\X$ be the reduced
subscheme of~$\AA^3_K$ consisting of the
eight points $p_1 = (0,0,0)$, $p_2 =(1,1,1)$, $p_3 = (-1,1,-1)$,
$p_4 = (2,4,8)$, $p_5 = (-2,4,-8)$, $p_6 = (3,9,27)$, $p_7 = (-3,9,-27)$, 
and $p_8 = (4,16,64)$ on the twisted cubic curve
$T=\{(t, t^2, t^3)\in \AA^3_K \mid t\in K\}$.

The reduced Gr\"obner basis with respect to the ${\tt DegRevLex}$
term ordering such that $z>y>x$ of the vanishing ideal~$I_\X$ is
\begin{align*}
\{ \; & x^2-y,\; xy-z,\; y^2-xz, 
yz^2-4xz^2-14z^2+56yz+49xz -196z -36y + 144x,\\
& z^3 - 30xz^2 +273 yz -820z +576x \; \}
\end{align*}
and hence we have
$\DF(I_\X) = \langle \, x^2,\, xy,\; y^2-xz,\, yz^2-4xz^2,\,
z^3-30xz^2 \rangle$.

Now we apply Algorithm~\ref{alg:CheckSGor}. To get a degree 
filtered $K$-basis of~$R_\X$, it suffices to take the terms
$\OO= \{1,\,  x,\,  y,\,  z,\,  xz,\,  yz,\,  z^2,\,  xz^2\}$
in the complement of the leading term ideal of~$I_\X$.
Then the Hilbert function of~$\OO$ is $(1,3,3,1)$, 
and we observe that it is symmetric. 

Next we determine the multiplication matrices 
corresponding to the multiplication by the elements of~$\OO$ 
on~$P/ \DF(I_\X)$ and combine their last columns to get
$$
V \;=\; \begin{pmatrix}
\;0\; & \;0\; & \;0\; & \;0\; & \;0\; & \;0\; & \;0\; & \;1\;   \\
0 & 0 & 0 &  0 & 0 & 0 &  1 & 0\\
0 & 0 & 0 &  0 & 0 & 1 &  4 & 0\\
0 & 0 & 0 &  0 & 1 & 4 & 30 & 0\\
0 & 0 & 0 &  1 & 0 & 0 &  0 & 0\\
0 & 0 & 1 &  4 & 0 & 0 &  0 & 0\\
0 & 1 & 4 & 30 & 0 & 0 &  0 & 0\\
1 & 0 & 0 &  0 & 0 & 0 &  0 & 0\\
\end{pmatrix}
$$
Since $\det(V)=1 \ne 0$, it follows that~$\X$
is a strict Gorenstein scheme.
\end{example}

\subsection{Computing the Strict Cayley-Bacharach Locus}

In the following we present algorithms to calculate the loci of all
strict Cayley-Bacharach schemes and all strict Gorenstein schemes
in~$\BOdf$. We use the following terminology.

\begin{definition}\label{def:StrictCBLocus}
Let $\OO=\{t_1,\dots,t_\mu\}$ be an order ideal in~$\mathbb{T}^n$.
\begin{enumerate}
\item[(a)] The set of all $K$-rational points $\Gamma=(\gamma_{ij})
\in K^{\mu\nu}$ of the border basis scheme~$\BO$
which represent a 0-dimensional strict Cayley-Bacharach scheme~$\X_\Gamma$ 
is called the {\bf set of strict Cayley-Bacharach points} of~$\BO$.

\item[(b)] The set of all $K$-rational points $\Gamma=(\gamma_{ij})
\in K^{\mu\nu}$ of the border basis scheme~$\BO$
which represent a 0-dimensional strict Gorenstein scheme~$\X_\Gamma$ 
is called the {\bf set of strict Gorenstein points} of~$\BO$.
\end{enumerate}
\end{definition}

Using the next algorithm, we see that there exists an open subscheme 
$\SCB^{\rm df}_\OO$  of~$\BOdf$ whose $K$-rational points are precisely 
the strict Cayley-Bacharach points of~$\BO$ which are contained in~$\BOdf$. 
This scheme is called the {\bf strict Cayley-Bacharach locus} in~$\BOdf$.
We need the following method for finding the generic
multiplication matrices on the associated graded rings.

\begin{remark}\label{GenHomMultMat}
Let $\OO=\{t_1,\dots,t_\mu\}$ be an order ideal in~$\mathbb{T}^n$,
let $\partial\OO=\{b_1,\dots,b_\nu\}$ be its border, and let
$G\df=\{g_1\df,\dots,g_\nu\df\}$ be the generic degree filtered
$\OO$-border prebasis. 
For every $K$-rational point $\Gamma=(\gamma_{ij})$
of~$\BOdf$, the 0-dimensional scheme~$\X_\Gamma$ has the
associated graded ring $\grF(R_{\X_\Gamma})\cong P/\DF(I_{\X_\Gamma})$.
Since $\X_\Gamma$ has a degree filtered $\OO$-border basis, we have
$\DF(I_{\X_\Gamma})=\langle \DF(g_1),\dots,\DF(g_\nu)\rangle$, where
$\DF(g_j)= b_j - \sum_{\{ i\mid \deg(t_i)=\deg(b_j)\}} \gamma_{ij}\, t_i$
(see~\cite{KR4}, Theorem 2.4).

Therefore all these associated graded rings are parametrized by the
$K$-algebra $U_\OO\hom = B_\OO\hom[x_1,\dots,x_n]/ 
\langle g_1\hom,\dots,g_\nu\hom \rangle$,
where $B_\OO\hom$ is the affine coordinate ring of the
homogeneous border basis scheme,  and where we have
the equality $g_j\hom = b_j - \sum_{\{ i\mid \deg(t_i)=
\deg(b_j)\}} c_{ij}\, t_i$ for $j=1,\dots,\nu$.
Thus the multiplication matrices of the associated graded rings
are parametrized by the generic homogeneous multiplication matrices
$\A_1\hom, \dots, \A_n\hom$ given in Remark~\ref{MultMatUhom}.
For an arbitrary element $f\in P$, the multiplication by~$f$
of the associated graded rings is therefore given by
$M^{\rm hom}_f=f(\A_1\hom, \dots, \A_n\hom)$ with respect to the
basis~$\OO$.
\end{remark}

Thus we are ready to compute the following subscheme of~$\BOdf$.

\begin{algorithm}{\upshape\bf (Computing the Strict Cayley-Bacharach Locus 
in $\BOdf$)}\label{alg:SCBLocDF}\\
Let $\OO=\{t_1,\dots,t_\mu\}$ be an order ideal in~$\mathbb{T}^n$,
and let $\Delta = \# \{i\in\{1,\dots,\mu\} \mid \deg(t_i)=\deg(t_\mu)\}$.
Consider the following sequence of instructions.
\begin{enumerate}
\item[(1)] Using Proposition~\ref{CharBOdf}.b, calculate
$I(\BOdf)=I(\BO)+I_\OO\df$.

\item[(2)] For $i=1,\dots,\mu$ use Remark~\ref{GenHomMultMat}
to compute the multiplication matrix $M_{t_i}\hom$ for the multiplication 
by~$t_i$ on~$U_\OO\hom$.

\item[(3)]
For  $j=1,\dots,\Delta$, form the matrix 
$V_j \in \Mat_\mu(K[C])$  whose $i$-th column is 
the $(\mu-\Delta+j)$-th column of $(M_{t_i}\hom)\tr$ for $i=1,\dots,\mu$.

\item[(4)] Form the block column matrix
$W=\Col(V_1, \dots, V_\Delta)$ and 
compute the ideal~$J$ in~$K[C]$ generated by the maximal 
minors of~$W$.

\item[(5)]  Return the ideal $I(\BOdf) + J$.
\end{enumerate}
This is an algorithm which computes an ideal in the 
ring~$K[C]$. This ideal defines a closed subscheme
$\NonSCB^{\rm df}_\OO$ whose $K$-rational points 
represent the 0-dimensional subschemes
of~$\AA^n_K$ which have a degree filtered $\OO$-border basis,
but are not strict Cayley-Bacharach schemes.
\end{algorithm}

\begin{proof}
A $K$-rational rational point $\Gamma=(\gamma_{ij})$ of~$\BOdf$
corresponds to a zero of the ideal $I(\BOdf)$ which is computed
in Step~(1). By Algorithm~\ref{alg:CheckSCBS}, the 
0-dimensional scheme $\X_\Gamma$ represented by~$\Gamma$
is a strict Cayley-Bacharach scheme
if and only if~$\Gamma$ is a zero of the ideal generated by the
maximal minors of~$W$, and this observation finishes the proof.
\end{proof}

In view of this algorithm, we see that $\SCB^{\rm df}_\OO = 
\BOdf \setminus \NonSCB^{\rm df}_\OO$ is an open subscheme of~$\BOdf$.
Let us calculate the scheme $\NonSCB^{\rm df}_\OO$ in the 
setting of Example~\ref{1-x-y-z-x^2}.

\begin{example}\label{1-x-y-z-x^2-cont}
Let $\OO$ be the order ideal $\OO=\{1,x,y,z,x^2\}$ in
$P=K[x,y,z]$. We apply Algorithm~\ref{alg:SCBLocDF}
to compute the ideal defining the subscheme $\NonSCB^{\rm df}_\OO$
of~$\BOdf$ and get $I(\BOdf) + \langle
c_{52}\, c_{53}^2 - 2\, c_{51}\,c_{53}\,c_{54} + c_{51}^2\, c_{55} 
+ c_{54}^2  - c_{52}\, c_{55} \rangle$. Hence the strict
Cayley-Bacharach locus and the Cayley-Bacharach locus in~$\BOdf$
are identical in this case. This is in agreement with~\cite{KLR1},
Theorems~6.8 and~6.12, since the Hilbert function of~$\OO$
is symmetric and satisfies $\Delta=1$.
\end{example}

In the following example, the Cayley-Bacharach and the strict 
Cayley-Bacharach locus differ.

\begin{example}\label{1-x-y-x2-xy-y2-x3}
Consider the order ideal $\OO=\{1,\,x,\,y,\,x^2,\, xy,\,y^2,\,
x^3\}$ in $P=K[x,y]$. Its border is $\partial\OO = 
\{ x^2y,\, xy^2,\, y^3,\, x^4,\, x^3y\}$.
Thus we have $\mu=7$, $\nu=5$, and the ideal
$I(\BO)$ has 28 generators. Since the order ideal~$\OO$
has the generic Hilbert function, we get $I(\BOdf)=I(\BO)$
by Proposition~\ref{CharMaxdeg}.

Now we use Algorithm~\ref{alg:CBLocDF} to compute
an ideal which defines~$\NonCB^{\rm df}_\OO$. 
We obtain the ideal $I(\BO) + \langle f \rangle$
where $f\in K[C]$ is a polynomial of the shape
$$
f \;=\; c_{43} c_{71}^5 \;-\; c_{53} c_{71}^4 c_{72} \;+\;
\cdots  \;+\; c_{42}c_{73}^2 
$$
Thus the Cayley-Bacharach locus in~$\BOdf$ is the complement
of a hypersurface section of~$\BOdf$.

Next we use Algorithm~\ref{alg:SCBLocDF} to compute
an ideal which defines~$\NonSCB^{\rm df}_\OO$. 
The result is the ideal $I(\BO)$, in agreement with~\cite{KLR1},
Thm.~6.12, which says that every strict Cayley-Bacharach scheme
with $\Delta_\X=1$ is a strict Gorenstein scheme. Since
the Hilbert function $(1,2,3,1)$ of~$\OO$ is not symmetric, this is 
impossible here. In other words, the strict Cayley-Bacharach locus 
in~$\BOdf$ is empty.
\end{example}

Finally, we combine Algorithm~\ref{alg:SCBLocDF} with a
check for the symmetry of the Hilbert function of~$\OO$
and get the following corollary. It shows that there exists 
an open subscheme $\SGor^{\rm df}_\OO$ of~$\BOdf$ whose $K$-rational points
are exactly those points in the set of strict
Gorenstein points of~$\BO$ which are contained in~$\BOdf$. 
This subscheme is called the {\bf strict Gorenstein locus} in~$\BOdf$.

\begin{corollary}{\upshape\bf (Computing the Strict Gorenstein Locus in 
$\BOdf$)} \label{cor:StrictGorLoc}\\
Let $\OO=\{t_1,\dots,t_\mu\}$ be an order ideal in~$\mathbb{T}^n$, and 
let $\rho=\deg(t_\mu)$. Then the following instructions 
define an algorithm which computes an ideal in~$K[C]$.
This ideal defines a closed subscheme $\NonSGor^{\rm df}_\OO$
whose $K\!$-rational points represent the 0-dimensional subschemes
of~$\AA^n_K$ which have a degree filtered $\OO$-border basis,
but are not strict Gorenstein schemes.
\begin{enumerate}
\item[(1)] Using Proposition~\ref{CharBOdf}.b, calculate 
$I(\BOdf)=I(\BO)+I_\OO\df$.

\item[(2)] If the Hilbert function $\HF_\OO$ 
is not symmetric, then return the ideal $I(\BOdf)$ and stop. 

\item[(3)] Otherwise, apply Algorithm~\ref{alg:SCBLocDF}
and return the ideal $I(\BOdf)+J$ it computes.
\end{enumerate}
\end{corollary}

\begin{proof}
A $K$-rational rational point $\Gamma=(\gamma_{ij})$ of~$\BOdf$
corresponds to a zero of the ideal $I(\BOdf)$ which is computed
in Step~(1). If the corresponding associated graded ring 
$\grF(R_{\X_\Gamma}) \cong P/\DF(I_{\X_\Gamma})$ is to be a Gorenstein
ring, its Hilbert function $(h_0,\dots,h_\mu)$ has to be
symmetric. Therefore Step~(2) correctly returns the ideal 
of~$\BOdf$ if that Hilbert function is not symmetric.

Now assume that the Hilbert function is symmetric.
By Algorithm~\ref{alg:CheckSGor}, we know that the 0-dimensional 
scheme $\X_\Gamma$ represented by~$\Gamma$ is a strict Gorenstein scheme
if and only if~$\Gamma$ is a zero of the ideal~$I(\BOdf)+J$, and 
this observation finishes the proof.
\end{proof}

Let us complete the discussion of Corollary~\ref{cor:StrictGorLoc}
by the following observation.

\begin{remark}\label{StrictGor=StrictCB}
Suppose that $\OO=\{t_1,\dots,t_\mu\}$ is an order ideal
such that only one term has the maximal degree
$\max \{ \deg(t_1),\dots,\deg(t_\mu)\}$.
Then~\cite{KLR1}, Thm.~6.12 implies
$\NonSGor^{\rm df}_\OO = \NonSCB^{\rm df}_\OO$.
Hence we may use either Algorithm~\ref{alg:SCBLocDF} 
or Corollary~\ref{cor:StrictGorLoc} to compute this subscheme.
\end{remark}

In the setting of Example~\ref{1-x-y-z-x^2}, the Hilbert
function of $\OO=\{1,x,y,z,x^2\}$ is symmetric, and therefore
the ideals defining the closed subschemes $\NonSCB^{\rm df}_\OO$ and
$\NonSGor^{\rm df}_\OO$ of~$\BOdf$ are identical. 
On the other hand, in the setting of Example~\ref{1-x-y-x2-xy-y2-x3},
the Hilbert function of~$\OO = \{ 1,x,y,x^2,xy,y^2,x^3\}$ is not symmetric and 
Corollary~\ref{cor:StrictGorLoc} returns an ideal 
which defines~$\BOdf$.

\bigbreak
%
%

\section{The Strict Complete Intersection Locus in $\BOdf$}
\label{The Strict Complete Intersection Locus in BODF}

Recall that a 0-dimensional ring of the form $K[x_1,\dots,x_n]/I$ with
a field~$K$ and a homogeneous 0-dimensional ideal~$I$
is called a {\bf graded complete intersection} if~$I$ can be generated
by a homogeneous regular sequence of length~$n$.
In our setting, the following version of this notion
will be considered.

\begin{definition}
Let $\X$ be a 0-dimensional subscheme of~$\AA^n_K$, let
$I_\X$ be the vanishing ideal of~$\X$ in~$P$, let $R_\X$ be the
affine coordinate ring of~$\X$, and let~$\F$ be the degree 
filtration of~$R_\X$.

Then the scheme~$\X$ is called a {\bf strict complete 
intersection scheme} if the associated graded ring
$\grF(R_\X)\cong P/\DF(I_\X)$ is a graded complete intersection.
\end{definition}

Various equivalent conditions and properties of this
notion are discussed in~\cite{KLR2}, Sections~4 and~5.
In particular, Algorithm~5.4 of~\cite{KLR2} 
provides a way to check whether~$\X$ is a strict
complete intersection scheme which uses the knowledge of a
degree filtered border basis.

This brings us to the topic of this section: to compute the
locus of strict complete intersection schemes inside the 
degree filtered border basis scheme. We begin again by fixing
the terminology.

\begin{definition}\label{def:StrictCILocus}
Let $\OO=\{t_1,\dots,t_\mu\}$ be an order ideal in~$\mathbb{T}^n$.
Then the set of all $K$-rational points $\Gamma=(\gamma_{ij})
\in K^{\mu\nu}$ of the border basis scheme~$\BO$
which represent a 0-dimensional strict complete intersection 
scheme~$\X_\Gamma$ is called the {\bf set of strict 
complete intersection points} of~$\BO$.
\end{definition}

The following algorithm computes a closed subscheme
$\NonSCI^{\rm df}_\OO$ of~$\BOdf$ such that the $K$-rational
points of the open subscheme $\SCI^{\rm df}_\OO = \BOdf
\setminus \NonSCI^{\rm df}_\OO$ are precisely
the strict complete intersection points in~$\BOdf$.
This open subscheme is called the {\bf strict complete intersection 
locus} in~$\BOdf$.

\begin{algorithm}{\upshape\bf (Computing the Strict CI Locus in 
$\BOdf$)}\label{alg:SCILocDF}\\
Let $\OO=\{t_1,\dots,t_\mu\}$ be an order ideal in~$\mathbb{T}^n$, and let
$\rho=\deg(t_\mu)$. Consider the following sequence of instructions.
\begin{enumerate}
\item[(1)] Using Proposition~\ref{CharBOdf}.b, calculate 
$I(\BOdf)=I(\BO)+I_\OO\df$.

\item[(2)] If the Hilbert function $\HF_\OO$ 
of~$\OO$ is not symmetric, then return
the ideal $I(\BOdf)$ and stop.

\item[(3)] Form the generic homogeneous $\OO$-border prebasis 
$G\hom=\{g_1\hom,\dots,g_j\hom\}$
and write $g_j\hom = \sum_{i=1}^n h_{ij}x_i$ with $h_{ij}\in
K[C][x_1,\dots,x_n]$ for $j=1,\dots,\nu$.

\item[(4)] Form the matrix~$W$ of size $n\times \nu$
whose columns are given by $(h_{1j},\dots,h_{nj})\tr$
for $j=1,\dots,\nu$.

\item[(5)] Let $k=\binom{\nu}{n}$. Calculate the minors
$f_1,\dots,f_k$ of order~$n$ of~$W$.

\item[(6)] Using border division by $G\hom$, write the residue classes 
$\bar{f}_1,\dots, \bar{f}_k \in U_\OO\hom$ as $B_\OO\hom$-linear
combinations $\bar{f}_j= \sum_{i=1}^\mu \bar{a}_{ij} t_i$
with $\bar{a}_{1j},\dots,\bar{a}_{\mu j}\in B_\OO\hom$ for
$j=1,\dots,k$.

\item[(7)] Let $C\hom =\{c_{ij} \mid \deg(t_i)=\deg(b_j)\}$.
For $i=1,\dots,\mu$ and $j=1,\dots,k$, choose $a_{ij}\in K[C\hom]$
which represents $\bar{a}_{ij}$ 
with respect to $B_\OO\hom \cong K[C\hom]/\bar{I}(\BOhom)$.
Return the ideal $J=I(\BOdf)+ \langle a_{ij} \mid i\in\{1,\dots,\mu\},
j\in\{1,\dots,k\}\rangle$ and stop.
\end{enumerate}
This is an algorithm which computes an ideal~$J$ in the ring~$K[C]$
which defines a closed subscheme $\NonSCI^{\rm df}_\OO$ of~$\BOdf$.
The $K\!$-rational points of this subscheme represent the 0-dimensional 
subschemes of~$\AA^n_K$ which have a degree filtered $\OO$-border basis,
but are not strict complete intersection schemes.
\end{algorithm}

\begin{proof}
In Step~(1) we calculate the ideal defining~$\BOdf$.
In Step~(2) we check whether the Castelnuovo function of
a 0-dimensional scheme $\X_\Gamma$ represented by a $K$-rational
point~$\Gamma$ of~$\BOdf$ is symmetric.
This is certainly a necessary condition, because strict
complete intersections are strict Gorenstein schemes.
If it is not satisfied, we return the ideal of~$\BOdf$.

Since the ideal~$J$ returned by the algorithm contains 
the ideal of $\BOdf$, we clearly compute a closed subscheme of~$\BOdf$.
Notice that the ideal~$J$ does not depend on the choice of the
representatives $a_{ij}$ in Step~(7), since the elements of
$\bar{I}(\BOhom)$ are contained in $I(\BOdf)$ by 
Proposition~\ref{IhomInIdf}.

A $K$-rational point~$\Gamma$ of~$\BOdf$
is a zero of the ideal~$J$ returned by the algorithm if and only if the
vanishing ideal$I_{\X_\Gamma}$ satisfies the conditions required
in Algorithm~5.4 of~\cite{KLR2}. Hence the scheme~$\X_\Gamma$ is not a
strict complete intersection if and only if the point~$\Gamma$
is a zero of the ideal~$J$, as was to be shown.
\end{proof}

In the following we illustrate this algorithm with
a couple of examples.

\begin{example}\label{1-x-y-xy-cont3}
As in Example~\ref{1-x-y-xy}, let $\OO$ be the order ideal 
$\OO = \{1,x,y,xy\}$ in $P=K[x,y]$. Then we have
$\partial \OO = \{ x^2,y^2,x^2y,xy^2\}$, and thus $\mu=\nu=4$.
The ideal $I(\BO)=I(\BOdf)$ was determined in Examples~\ref{1-x-y-xy} 
and~\ref{1-x-y-xy-cont2}.a. It is generated by 12 quadratic
equations in 16 indeterminates.

Since the Hilbert function $(1,2,1,0,\dots)$ of~$\OO$ is symmetric, we can use
Corollary~\ref{cor:StrictGorLoc} to compute the ideal defining
the closed subscheme $\NonSGor^{\rm df}_\OO$ of~$\BOdf$
and get $I(\BOdf) + \langle c_{41}c_{42}-1 \rangle$.

Next we apply Algorithm~\ref{alg:SCILocDF} to compute the ideal
defining $\NonSCI^{\rm df}_\OO$. The matrix ~$W$ in Step~(4) is
$$
W \;=\; \begin{pmatrix}
-x_2 c_{41} +x_1 & \;-x_2 c_{42}\; & \;x_1x_2\; & \;x_2^2\; \\
0 & x_2 & 0 & 0
\end{pmatrix}
$$
This matrix has three non-zero maximal minors, namely
$f_1= -x_2^2 c_{41} + x_1x_2$, $f_2=-x_1x_2^2$, and $f_3=-x_2^3$.
In Step~(6) we apply border division by 
$$
G\hom \;=\; \{ \, x_1^2 - x_1x_2 c_{41},\; x_2^2 - x_1x_2c_{42},\;
x_1^2x_2,\; x_1x_2^2\,\} 
$$ 
to these polynomials. Only the division of~$f_1$ by $G\hom$ yields
a non-zero result, namely $x_1x_2(-c_{41}c_{42}+1)$. 
Consequently, the algorithm returns the ideal $I(\BO) + \langle 
c_{41}c_{42}-1\rangle$ in Step~(7).

Thus the ideals which define $\NonSGor^{\rm df}_\OO$ and
$\NonSCI^{\rm df}_\OO$ are identical.
This is in agreement with the fact that in codimension~2 the
notions of a Gorenstein ring and a complete intersection 
are equivalent.
\end{example}

In order to find a case where the loci of strict Gorenstein and
strict complete intersection scheme differ, we have to consider 0-dimensional 
schemes in~$\AA^n_K$ with $n\ge 3$. The following case is one of
the easiest ones.

\begin{example}\label{1-3-3-1}
Consider the order ideal $\OO=\{1,x,y,z,x^2,xy,xz,x^3\}$
in $P=K[x,y,z]$. Its Hilbert function $\HF_\OO=(1,3,3,1,0,\dots)$ is
symmetric. The border of~$\OO$ is
$\partial\OO =  \{\, y^2,\, yz,\, z^2,\, x^2y,\, x^2z,\, xy^2,\, 
xyz,\, xz^2,\, x^4,\, x^3y,\, x^3z \,\}$.
Consequently, we have $\mu=8$ and $\nu=11$,
and the ring~$K[C]$ has 88 indeterminates.
Using Proposition~\ref{CharBOdf}.b, we calculate 
$I(\BOdf)=I(\BO)+I_\OO\df$. The resulting ideal
has 147 generators, 3 of which are indeterminates,
and the remaining 144 are quadratic polynomials.

Using Corollary~\ref{cor:StrictGorLoc}, we compute an
ideal which defines $\NonSGor^{\rm df}_\OO$ and get
$I(\BOdf) + \langle f \rangle$, where~$f$ is a polynomial
of the form $f=g^2$ with
$$
g \;=\; c_{85}^2 c_{86} - 2 c_{84} c_{85} c_{87} + c_{48}^2 c_{88}
+ c_{87}^2 - c_{86} c_{88}
$$
Thus the strict Gorenstein locus in~$\BOdf$ is the complement of 
the hypersurface section defined by~$f$. Notice that the polynomials~$f$
and~$g$ are homogeneous of arrow degree zero.

Next we apply Algorithm~\ref{alg:SCILocDF} to calculate
an ideal which defines $\NonSCI^{\rm df}_\OO$. 
The matrix~$W$ in Step~(4) is 
$$
W \;=\; \begin{pmatrix}
\;p_1\; & \;p_2\; & \;p_3\; & \;p_4\; & \;p_5\; & \;p_6\; & 
   \;p_7\; & \;p_8\; & \;x_1^3\; & \;x_1^2 x_2\; & \;x_1^2x_3\\
x_2 & x_3 & 0   & 0 & 0 & 0 & 0 & 0 & 0 & 0 & 0\\
0   & 0   & x_3 & 0 & 0 & 0 & 0 & 0 & 0 & 0 & 0
\end{pmatrix}
$$
where $p_1 = -x_1 c_{51} -x_2 c_{61} -x_3 c_{71}$, $p_2 = -x_1 c_{52} -x_2
c_{62} -x_3 c_{72}$, $p_3 = -x_1 c_{53} - x_2 c_{63} -x_3 c_{73}$,
$p_4= -x_1^2 c_{84} + x_1x_2$, $p_5 = -x_1^2c_{85} +x_1x_3$,
$p_6 = -x_1^2 c_{86} + x_2^2$, $p_7 = -x_1^2 c_{87} + x_2x_3$, and
$p_8 = -x_1^2c_{88} + x_3^2$. This matrix has 17 non-zero maximal minors,
and only one of them yields a non-zero remainder after division 
by~$G\hom$, namely $h\cdot x^3$, where~$h$ is the
polynomial
\begin{align*}
h \;=\; &-c_{52}c_{61}c_{85} +c_{51}c_{62}c_{85} -c_{53}c_{71}c_{85}
+c_{52}c_{72}c_{85} -c_{63}c_{71}c_{87} +c_{62}c_{72}c_{87}\\
& +c_{62}c_{71}c_{88} -c_{61}c_{72}c_{88} + c_{72}^2 c_{88} 
-c_{71}c_{73}c_{88} + c_{52}c_{87} -c_{51}c_{88}
\end{align*}

Altogether, the result is the ideal $I(\BOdf) +\langle h \rangle$
which defines $\NonSCI^{\rm df}_\OO$.
At this point we may check that $I(\BOdf) + \langle f\rangle
\subsetneq I(\BOdf) + \langle g\rangle =
I(\BOdf) + \langle h\rangle$. 
This shows that the locus of strict complete intersections 
in~$\BOdf$ is properly contained
in the locus of strict Gorenstein schemes in~$\BOdf$.
As Example~\ref{ex:8pointscubic} and~\cite{KLR2}, Example~4.6 indicate, 
eight points on a suitably chosen twisted cubic curve yield a point in the
strict Gorenstein locus of~$\BOdf$ which is not contained in the
strict complete intersection locus.
\end{example}

\bigbreak
%
%

\section{The Hilbert Stratification of~$\BO$}
\label{The Hilbert Stratification of BO}

In general, the border basis scheme~$\BO$ contains
$K\!$-rational points which represent 0-dimensional schemes
having different affine Hilbert functions. In this section
we describe and calculate the stratification of~$\BO$
determined by these affine Hilbert functions. 

As in the preceding sections, let $K$ be a field, 
let $P=K[x_1,\dots, x_n]$,
and let $\OO=\{t_1,\dots,t_\mu\}$ be an order ideal in~$\mathbb{T}^n$
with border $\partial\mathcal{O}=\{b_1,\dots,b_\nu\}$.
Recall that we always order~$\OO$ such that
$\deg(t_1) \le \cdots \le \deg(t_\mu)$. In particular,
we have $t_1=1$ (see Assumption~\ref{orderOO}). 
In this setting, we introduce the following terminology
(see also Definition~\ref{def:affineHF}).

\begin{definition}\label{succeq}
Given a $K$-rational point $\Gamma=(\gamma_{ij})$ of~$\BO$,
let $G_\Gamma$ be the corresponding border basis, and let
$R_\Gamma = P/\langle G_\Gamma\rangle$ be the
affine coordinate ring of the
0-dimensional affine scheme~$\X_\Gamma$ represented by~$\Gamma$.
\begin{enumerate}
\item[(a)]  The affine Hilbert function of~$R_\Gamma$ is
called the affine Hilbert function {\bf associated to~$\Gamma$}.
The values of the affine Hilbert function associated to~$\Gamma$
will be denoted by
$$
\HF^a_{R_\Gamma} \;=\; (H_0,\, H_1,\, H_2, \dots)
$$
where $H_i=\HF^a_{R_\Gamma}(i)$ for $i\ge 0$.

\item[(b)] The values of the Castelnuovo function of~$R_\Gamma$
will be denoted by
$$
\Delta\HF^a_{R_\Gamma} = (h_0,\, h_1,\, h_2,\dots)
$$
where $h_0=1$ and $h_i=H_i-H_{i-1}$ for $i\ge 1$.

\item[(c)] We say that a sequence of non-negative integers
$(H_0,H_1,\dots)$ {\bf dominates} a sequence of non-negative integers
$(H_0', H_1',\dots)$
if we have $H_i\ge H_i'$ for all $i\ge 0$. In this case
we also write $(H_0,H_1,\dots) \succeq (H_0',H_1',\dots)$.

\item[(d)] Given $n,\mu\in \mathbb{N}_+$, a sequence $\H=(H_0,H_1,\dots)$
of non-negative integers is called {\bf $(n,\mu)$-admissible}
if there exists a 0-dimensional affine subscheme~$\X$ 
in~$\mathbb{A}^n_K$ of length~$\mu$ such that $\H=\HFa_\X$.

\item[(e)]  Given $n,\mu\in \mathbb{N}_+$, the sequence 
$\HF^{(n,\mu)}=(H_0,H_1,\dots)$ such that we have $H_i = \min \{\mu,\;
\HFa_P(i) \}$ for all $i\ge 0$ is called the {\bf generic affine Hilbert
function} for a scheme of length~$\mu$ in~$\mathbb{A}^n_K$.
\end{enumerate}
\end{definition}

In~\cite{KR2}, Section~5.5, precise formulas are given which
characterize admissible Hilbert functions. Also, note that 
from~\cite{KR2}, Theorem~5.5.32, it follows that~$\HF^{(n,\mu)}$
is admissible.

\begin{proposition}\label{dominatingHF}
Let $\Gamma$ be a $K$-rational point of~$\BO$, and
let $R_\Gamma$ be the affine coordinate ring
of the 0-dimensional affine scheme represented by~$\Gamma$.
Then we have 
$\HF^{(n,\mu)} \succeq  \HF^a_{R_\Gamma}  \succeq \HF^a_\OO$.
\end{proposition}

\begin{proof}
The generic affine Hilbert function for a scheme
of length~$\mu$ dominates the affine Hilbert function 
of~$\X_\Gamma$, since its values exhibit the maximal possible growth.

On the other hand, the  affine Hilbert function of~$R_\Gamma$ 
dominates~$\HFa_\OO$, since the terms of~$\OO$ are linearly 
independent in~$R_\Gamma$.
\end{proof}

Here is an easy example of an order ideal~$\OO$ such that 
there are 0-dimensional schemes represented by rational points in $\BO$
having different associated affine Hilbert functions.

\begin{example}\label{differentHF}
Let $P=K[x,y]$, and let $\OO = \{1,y, y^2\} \subseteq \mathbb{T}^2$. 
Then we have $\HF^a_\OO = (1,2, 3,3,\dots)$.
Now let $I=\langle x, y^3\rangle$ and $J=\langle x-y^2, y^3\rangle$.
The set $\{x, y^3\}$ is the reduced $\tt Lex$-Gr\"obner basis of~$I$,
and the set$\{x-y^2,\, y^3\}$ 
is the reduced $\tt Lex$-Gr\"obner basis of~$J$.
Consequently, the rings $P/I$ and $P/J$ are coordinate rings
of 0-dimensional affine schemes represented by $K$-rational 
points in~$\BO$. Then we have $\HF^a_{P/I} = (1,2,3,3,\dots)$
and $\HF^a_{P/J}= (1,3,3,\dots)$.

From the growth conditions for Hilbert functions in~\cite{KR2},
Section~5.5, it follows that $(1,2,3,3,\dots)$ and ($1,3,3,\dots)$ 
are the only admissible affine Hilbert functions for a 0-dimensional 
scheme of length 3 in $\mathbb{A}^2_K$. Note that
$\HF^{(2,3)}=\HF^a_{P/J} \succ \HF^a_{P/I}$. 
\end{example}

In this example all admissible affine Hilbert 
functions are totally ordered with respect to~$\succ$.
The following example shows that this is not true in general.

\begin{example}\label{noncomparableHF}
Consider all admissible affine
Hilbert functions with $n=3$ and $\mu=11$. We have
the relations
$(1, 4, 10, 11, 11,\dots) \succ (1, 4, 9, 10, 11, 11,\dots)$ 
and $(1, 4, 10, 11, 11,\dots) \succ (1, 4, 8, 11, 11,\dots)$, 
but the sequences $(1,4,9, 10, 11, 11,\dots)$ and $(1, 4, 8, 11, 11,\dots)$  
are not comparable with respect to~$\succ$.
\end{example}

Our next goal is to describe the locus in~$\BO$ of all
0-dimensional affine schemes whose affine Hilbert function
is a given function~$\H = (H_0,H_1,\dots)$.
For this purpose we first determine the equations defining
a closed subscheme of~$\BO$ such that its $K$-rational points represent
those 0-dimensional affine schemes whose affine Hilbert function 
is dominated by~$\H$. 
As a preliminary step, we determine the closed subscheme of~$\BO$ 
corresponding to all $K$-rational points whose associated affine Hilbert 
function satisfies $H_k\le N$ for fixed given numbers $k\ge 0$ and $N>0$
as follows.

\begin{algorithm}\label{alg:oneHi}
Let $\OO=\{t_1,\dots,t_\mu\}$ be an order ideal,
let $k\ge 0$, and let $N\ge 1$. Consider the following
sequence of instructions.
\begin{enumerate}
\item[(1)] Let $u_1,\dots,u_m \in \mathbb{T}^n$
be the terms of degree $\le k$.

\item[(2)] For $\ell=1,\dots,m$, compute the matrix
$\mathcal{U}_\ell = u_\ell(\A_1,\dots,\A_n)$
in the ring $\Mat_\mu(K[C])$, where $\A_1,\dots,\A_n$
are the generic multiplication matrices for~$\mathcal{O}$.

\item[(3)] Form the matrix $\mathcal{M}$ in
$\Mat_{\mu, m}(K[C])$ whose $\ell$-th column is the
first column of~$\mathcal{U}_\ell$ for $\ell=1,\dots,m$.

\item[(4)] Compute the ideal~$D$ of minors of size $N+1$
of~$\mathcal{M}$ and return $J=I(\BO)+D$.
\end{enumerate}
This is an algorithm which computes an ideal~$J$ in~$K[C]$
such that $\mathcal{Z}(J)$ is the closed subscheme of~$\BO$
corresponding to all $K$-rational points whose associated 
affine Hilbert function satisfies $H_k\le N$.
\end{algorithm}

\begin{proof}
The universal border basis family~$U_\mathcal{O}$ is a free
$B_\mathcal{O}$-module with basis~$\mathcal{O}$, and for
$\nu=1,\dots,n$ the generic multiplication matrix $\A_\nu$
describes the multiplication by~$x_\nu$ in this basis.
Hence, for $\ell=1,\dots,m$, the matrix $\mathcal{U}_\ell$
describes the multiplication by~$u_\ell$ in the basis~$\mathcal{O}$.
Thus $t_1=1$ implies that the first column of~$\mathcal{U}_\ell$
contains the coordinates of~$u_\ell$ in the basis~$\mathcal{O}$.
Now the claim follows from the remark that we have
$H_k\le N$ for the affine Hilbert function associated to a $K$-rational
point~$\Gamma$ of~$\BO$ if and only if at most~$N$ terms
in the set $\{u_1,\dots,u_m\}$ are $K$-linearly independent
modulo $\langle G_\Gamma\rangle$.
\end{proof}

Let us check the results of this algorithm in a simple case.

\begin{example}\label{1-x-y-x2-x3}
Let $\OO$ be the order ideal $\OO = \{1,x,y,x^2,x^3\}$ in~$\mathbb{T}^2$.
The border of~$\OO$ is $\partial{\OO} = \{ xy,y^2,x^2y,x^4, x^3y \}$.
We have $\mu=5$, and the generic affine Hilbert function is $(1,3,5,5,\dots)$.

When we use the algorithm to compute the closed subscheme of~$\BO$
corresponding to $H_2\le 4$, we obtain the ideal
$I(\BO)+ \langle c_{51},c_{52}\rangle$. Notice that this ideal equals
the ideal $I(\BOdf)$. This is due to the fact that
the only possible affine Hilbert function of a $K$-rational point~$\Gamma$
with $\HF_{R_\Gamma}(2)\le 4$ is $(1,3,4,5,5,\dots)$, and thus
these are exactly the points for which~$\OO$ is a degree filtered $K$-basis
of~$R_\Gamma$.

However, if we use the algorithm to compute the closed subscheme
of~$\BO$ corresponding to $H_2\le 3$, we obtain the unit ideal 
$\langle 1\rangle$. Although there exists an affine Hilbert function
satisfying this constraint, namely $(1,2,3,4,5,5,\dots)$,
that affine Hilbert function is not attained by any 0-dimensional 
affine scheme~$\X$ whose vanishing ideal has an $\OO$-border basis, because
$\{1,x,y\} \subset \OO$ implies that there is no linear polynomial
in the vanishing ideal of~$\X$.
\end{example}

The above algorithm allows us to describe the following
subscheme of~$\BO$.

\begin{definition}\label{def:BOHbar}
Let $\OO=\{t_1,\dots,t_\mu\}$ be an order ideal in~$\mathbb{T}^n$,
and assume that $\H=(H_0,H_1,\dots)$ is an $(n,\mu)$-admissible affine 
Hilbert function. Then the closed subscheme of~$\BO$ 
which is the closure of the set of $K$-rational points~$\Gamma$
for which~$\HF^a_{R_\Gamma}$ is dominated by~$\H$ is called 
the {\bf $\Hbar$-subscheme} of~$\BO$ 
and is denoted by~$\BO(\Hbar)$.
\end{definition}

An ideal defining $\BO(\Hbar)$ can be calculated as follows.

\begin{algorithm}{\upshape\bf (Computing the $\Hbar$-Subscheme of~$\BO$)} 
\label{alg:BOHbar}\\
Let $\mathcal{O}=\{t_1,\dots,t_\mu\}$ be an order ideal, and let
$\H=(H_0,H_1,\dots)$ be a sequence of non-negative integers 
which dominates~$\HFa_\OO$.
The following instructions define an algorithm
which computes an ideal $I(\BO(\Hbar))$ in $K[C]$
that defines~$\BO(\Hbar)$.
\begin{enumerate}
\item[(1)] Let $\rho=\min\{i\ge 0 \mid H_i=\mu\}$.

\item[(2)] For $i=1,\dots,\rho-1$, use Algorithm~\ref{alg:oneHi}
to compute an ideal~$J_i$ in~$K[C]$ which defines the closed 
subscheme of~$\BO$ whose $K$-rational points~$\Gamma$ satisfy
$\HF^a_{R_\Gamma}(i)\le H_i$.

\item[(3)] Return the ideal $J_1+\cdots+ J_{\rho-1}$.

\end{enumerate}
\end{algorithm}

\begin{proof}
The finiteness of this algorithm is clear.
The correctness  follows from Algorithm~\ref{alg:oneHi}, since an affine
Hilbert function~$\HFa_{R_\Gamma}$ is dominated by~$\H$
if and only if each of its values is less than or equal
to the corresponding value of~$\H$.
\end{proof}

Let us apply this algorithm in the setting of Example~\ref{1-x-y-x2-x3}.

\begin{example}\label{1-x-y-x2-x3-cont}
In~$\mathbb{T}^2$ we consider the  order ideal $\OO=\{1,x,y,x^2,x^3\}$,
and we let $\mathcal{H}=(1,3,5,5,\dots)$.
When we compute the ideal $I(\BO(\Hbar))$ using
the algorithm, we get~$I(\BO)$. This means that we have
$\BO(\Hbar)=\BO$ here, in agreement with the observation
that~$\H$ is the generic affine Hilbert function for $n=2$ and~$\mu=5$.
\end{example}

\begin{definition}\label{def:BOH}
Let $\mathcal{O}=\{t_1,\dots,t_\mu\} \in \mathbb{T}^n$ 
be an order ideal, and let $\H=(H_0,H_1,\dots)$ be an 
$(n,\mu)$-admissible affine Hilbert function.
The open subscheme of~$\BO(\Hbar)$ which is the 
complement in~$\BO(\Hbar)$ of the closure of the set  
of all \hbox{$K$-rational} points~$\Gamma$ for which the associated affine Hilbert function
$\HF^a_{R_\Gamma}$ is strictly dominated by~$\H$, is called the {\bf $\H$-subscheme} 
of~$\BO$ and is denoted by~$\BO(\H)$.
\end{definition}

In other words, the $K$-rational points~$\Gamma$ of~$\BO(\H)$
correspond to 0-dimensional rings $R_\Gamma$ whose
affine Hilbert function is~$\H$.
The following algorithms allows us to clarify and compute 
the structure of $\BO(\H)$.

\begin{algorithm}{\upshape \bf (Computing the $\H$-Subscheme of~$\BO$)} 
\label{alg:BOH}\\
Let $\mathcal{O}=\{t_1,\dots,t_\mu\}$ be an order ideal, and let
$\H=(H_0,H_1,\dots)$ be a sequence of non-negative integers 
which dominates~$\HFa_\OO$.
The following instructions define an algorithm
which computes an ideal $I(Z_\OO(\Hbar))$ in $K[C]$ which
contains $I(\BO(\Hbar))$ and defines
a closed subscheme $Z_\OO(\Hbar)$ of~$\BO(\Hbar)$
with $\BO(\H) = \BO(\Hbar) \setminus Z_\mathcal{O}(\Hbar)$.
\begin{enumerate}
\item[(1)] Let $\rho=\min\{i\ge 0 \mid H_i=\mu\}$.

\item[(2)] For $i=1,\dots,\rho$, let $\H'_i =
(1,H_1,\dots,H_{i-1},H_i-1,H_{i+1},H_{i+2},\dots)$.

\item[(3)] For $i=1,\dots,\rho$, check whether
$H_i-1 \ge \HFa_\OO(i)$. If this is the case,
use Algorithm~\ref{alg:BOHbar}
to compute the ideal $J_i$ defining the
${\overline{\H'_i}}^{\mathstrut}$-subscheme of~$\BO$.
Otherwise, let $J_i=\langle 1\rangle$.

\item[(4)] Compute the ideal $J_1\cap\cdots\cap J_\rho$
and return it.
\end{enumerate}
\end{algorithm}

\begin{proof}
Note that, for a $K$-rational point $\Gamma\in \BO(\Hbar)$, 
an affine Hilbert function $\HFa_{R_\Gamma}$ is strictly
smaller than~$\H$ if and only if at least one of its values 
is strictly smaller than the corresponding value of~$\H$.
Since~$\H$ dominates the affine Hilbert function of~$\OO$
and since~$\OO$ is $K$-linearly independent in~$R_\Gamma$,
there can be no such point~$\Gamma$ if $H_i-1<\HFa_\OO(i)$.
Thus we can drop the corresponding ideals $J_i$ from the intersection.
Finally, we note that each ideal $J_i$ contains $I(\BO(\Hbar))$,
and therefore also the resulting intersection ideal does so.
\end{proof}

Let us apply this algorithm to the setting of 
Example~\ref{1-x-y-x2-x3-cont}.

\begin{example}\label{1-x-y-x2-x3-cont-again}
As in Example~\ref{1-x-y-x2-x3-cont}, let
$\OO=\{1,x,y,x^2,x^3\} \subseteq \mathbb{T}^2$,
and let $\H=(1,3,5,5,\dots)$. We use Algorithm~\ref{alg:BOH} 
to calculate the ideal $I(Z_\OO(\Hbar))$
defining the complement of~$\BO(\H)$ in~$\BO(\Hbar)$. 
Since $3=\HFa_\OO(1)$, we do not need to compute~$J_1$.
In order to find~$J_2$, it suffices to take $I(\BO)$ and to add
the ideal computed by Algorithm~\ref{alg:oneHi} for $i=2$ and $N=4$.
The result is $I(\BO) + \langle c_{51},\, c_{52}\rangle$, as
noted in Example~\ref{1-x-y-x2-x3}.
Altogether, we obtain $I(Z_\OO(\Hbar))=
I(\BO) + \langle c_{51},\, c_{52}\rangle$.
\end{example}

Notice that the algorithm correctly returns $I(\BO(\Hbar))$
if the function~$\H$ is not $(n,\mu)$-admissible. The following
example illustrates this behaviour.

\begin{example}
In~$\mathbb{T}^2$, consider the order ideal $\OO=\{1,\, x,\, x^2,\, x^3\}$.
The sequence $\H=(1,2,4,4,\dots)$ dominates the affine Hilbert function
$\HFa_\OO = (1,2,3,4,4,\dots)$ of~$\OO$.
Algorithm~\ref{alg:BOH} yields the ideals $J_1=\langle 1 \rangle$
and $J_2=I(\BO(\overline{\H'_2}))=I(\BO(\Hbar))$, and hence 
returns $I(\BO(\Hbar))$. Thus it follows that 
$\BO(\H)=\BO(\Hbar) \setminus \BO(\Hbar)=\emptyset$.
This is in agreement with the fact that~$\H$ is not $(2,4)$-admissible.
\end{example}

\begin{corollary}\label{HilbertStrat}
Let $\mathcal{O}=\{t_1,\dots,t_\mu\}$ be an order ideal
in~$\mathbb{T}^n$.
\begin{enumerate}
\item[(a)] Let $\H$ be an $(n,\mu)$-admissible affine Hilbert function
with $\HF^{(n,\mu)} \succeq \H \succeq \HFa_\OO$.
Then $\BO(\H)$ is an open subscheme  
of~$\BO(\Hbar)$ and a locally closed subscheme of~$\BO$.
It is called the {\bf $\H$-Hilbert stratum} of~$\BO$.

\item[(b)] Let $\mathbb{H}_\OO$ be the set of all
$(n,\mu)$-admissible affine Hilbert functions~$\H$
with $\HF^{(n,\mu)} \succeq \H \succeq \HFa_\OO$.
Then we have a disjoint union
$\BO = \bigcup_{\H \in \mathbb{H}^\OO} \BO(\H)$.
This is called the {\bf Hilbert stratification} of~$\BO$.
\end{enumerate}
\end{corollary}

\begin{proof}
The claims in~(a) follow from Algorithm~\ref{alg:BOH}
and Algorithm~\ref{alg:BOHbar}.
Claim~(b) follows from Proposition~\ref{dominatingHF}.
\end{proof}

The order ideals making up the Hilbert stratification
of~$\BO$ can be determined as follows.

\begin{remark}
Let $\OO=\{t_1,\dots,t_\mu\}$ be an order ideal 
in~$\mathbb{T}^n$. Consider the
recursively defined function ${\tt AllStrata}(\OO,\H)$
defined by the following steps.

\begin{enumerate}
\item[(1)] Let $S$ be the set consisting of~$\H$.

\item[(2)] Let $\H=(H_0,H_1,\dots)$, and let
$\rho=\min \{ i\ge 0 \mid H_i=\mu\}$.

\item[(3)] For $i=1,\dots,\rho$, check whether
$H_i-1\ge \HFa_\OO(i)$. If this is the case,
compute ${\tt AllStrata}(\OO,\H')$, where
$\H'=(H_0,\dots,H_{i-1},H_i-1,H_{i+1},\dots)$,
append it to~$S$, and remove doubles in~$S$.

\item[(4)] If $\H=\HFa_\OO$ then return~$S$ and stop.
\end{enumerate}
Let $S$ be the result of calling ${\tt AllStrata}(\OO,\HF^{(n,\mu)})$.
For every sequence $\H$ in~$S$, use~\cite{KR2}, Theorem 5.5.32.a, to
check whether~$\H$ is admissible. If this is not the case, remove~$\H$
from~$S$. The resulting set~$S$ is the set of all $(n,\mu)$-admissible affine 
Hilbert functions with $\HF^{(n,\mu)} \succeq \H \succeq \HFa_\OO$.
\end{remark}

\bigbreak
%
%

\section{An Open Covering of the Hilbert Stratum}
\label{An Open Covering of the Hilbert Stratum}

In several of the algorithms in the next sections
it will be necessary not only to fix the Hilbert function
of $R=P/I$, but also a degree filtered $K$-basis of~$R$.

As above, let $K$ be a field, and let
$\mathcal{O}=\{t_1,\dots,t_\mu\}$ be an order ideal
with border $\partial\mathcal{O}=\{b_1,\dots,b_\nu\}$.
In this setting, we introduce the following subschemes
of~$\BO$.

\begin{definition}
Let $\mathcal{O}=\{t_1,\dots,t_\mu\}$ be an order ideal.
\begin{enumerate}
\item[(a)] Let $\mathcal{O}'=\{t'_1,\dots,t'_\mu\}$
be a further order ideal containing~$\mu$ terms.
We say that~$\mathcal{O}'$ is a {\bf degree filtered order ideal
for~$\mathcal{O}$} if there exists an ideal~$I\subseteq P$
which has an $\mathcal{O}$-border basis and for which
$\mathcal{O}'$ is a degree filtered $K$-basis of~$P/I$.

\item[(b)] Let $\mathcal{O}'$ be a degree filtered order
ideal for~$\mathcal{O}$, and let $\H=(H_0,H_1,\dots)$ be the
affine Hilbert function of~$\OO'$. The open subscheme of~$\BO(\H)$
whose $K$-rational points represent the 0-dimensional affines 
schemes~$\X$ such that $\mathcal{O}'$ is a degree filtered 
$K\!$-basis of~$R_\X$ (and hence~$\H=\HFa_\X$) is called 
the {\bf $\mathcal{O}'$-DFB subscheme} 
and is denoted by $\BO^{\rm dfb}(\mathcal{O}')$.
\end{enumerate}
\end{definition}

The next algorithm shows that $\BO^{\rm dfb}(\mathcal{O}')$ is in
fact an open subscheme of~$\BO(\Hbar)$ which is contained 
in~$\BO(\H)$.

\begin{algorithm}{\upshape\bf (Computing a DFB 
Subscheme)}\label{alg:DFB}\\
Let $\mathcal{O}=\{t_1,\dots,t_\mu\}$ be an order ideal 
in~$\mathbb{T}^n$, let $\H$ be an
$(n,\mu)$-admissible Hilbert function which dominates $\HFa_\OO$,
and let $\mathcal{O}'=\{t'_1,\dots,t'_\mu\}$ be an
order ideal such that $\HFa_{\OO'}=\H$.
Consider the following instructions.
\begin{enumerate}
\item[(1)] Using Algorithm~\ref{alg:BOHbar}, compute an ideal
$I(\BO(\Hbar))$ which defines $\BO(\Hbar)$.

\item[(2)] Using Algorithm~\ref{alg:BOH}, applied to~$\H$, 
compute an ideal $I(Z_\OO(\Hbar))$
which defines the closed subscheme $Z_\OO(\Hbar)$
of~$\BO(\Hbar)$.

\item[(3)] For $i=1,\dots,\mu$, calculate the matrix
$t'_i(\A_1,\dots\A_n)$, where $\A_1,
\dots,\A_n$ are the generic multiplication matrices
for~$\mathcal{O}$.

\item[(4)] Form the matrix $T'$ in~$\Mat_\mu (K[C])$
consisting of the first columns of the matrices $t'_i(\A_1,\dots,
\A_n)$ for $i=1,\dots,\mu$.

\item[(5)] Return the triple $(I(\BO(\Hbar)),
I(Z_\OO(\Hbar)),\det(T'))$.
\end{enumerate}

This is an algorithm which computes a triple
$(I(\BO(\Hbar)), I(Z_\OO(\Hbar)),\det(T'))$ such that,
if we let $I^{\rm dfb}_{\OO\OO'} = 
I(Z_\OO(\Hbar)) \cap \langle \det(T') \rangle $,
then $J_{\OO'}=I^{\rm dfb}_{\OO\OO'} + I(\BO(\Hbar))$
satisfies $\BO^{\rm dfb}(\OO')=
\BO(\Hbar)\setminus \mathcal{Z}(J_{\OO'})
= \BO(\H) \setminus \mathcal{Z}(\det(T'))$.
\end{algorithm}

\begin{proof}
Since $\H=\HFa_{\OO'}$, the set~$\mathcal{O}'$ is a degree filtered
$K$-basis of~$R_{\X_\Gamma}$ for some closed
point~$\Gamma$ of~$\BO$ representing a 0-dimensional
affine scheme~$\X_\Gamma$ if and only if it is a $K$-basis.

For $i=1,\dots,\mu$, the first column of the matrix
$t'_i(\A_1,\dots,\A_n)$ contains
the coordinates of~$t'_i$ in the $K$-basis $\mathcal{O}$
of~$U_\mathcal{O}$, because this matrix is the
multiplication matrix of~$t'_i$ and we assumed $t_1=1$.
Hence, at a $K$-rational point $\Gamma=(\gamma_{ij})$ of~$\BO$,
we have $\det(T')(\gamma_{ij})\ne 0$ if and only
if~$\mathcal{O}'$ is a $K$-basis of~$R_\Gamma$.
Since the set $\BO(\H)$ is the complement of
$\mathcal{Z}(I(Z_\OO(\Hbar)))$ in~$\BO(\Hbar)$,
the proof is complete.
\end{proof}

The sets $\BO^{\rm dfb}(\OO')$ form an open covering of~$\BO(\H)$,
as the next corollary shows.

\begin{corollary}\label{OpenCover}
In the setting of the algorithm, let $\OO'_1,\dots,
\OO'_\ell$ be the order ideals in~$\mathbb{T}^n$
with affine Hilbert function~$\H$.
\begin{enumerate}
\item[(a)] For $i=1,\dots,\ell$, the set
$\BO^{\rm dfb}(\OO'_i)$ is an open subscheme
of~$\BO(\Hbar)$ which is contained in~$\BO(\H)$.

\item[(b)] The subschemes $\BO^{\rm dfb}(\OO'_1),\dots,
\BO^{\rm dfb}(\OO'_\ell)$ form an open covering of~$\BO(\H)$.
\end{enumerate}
\end{corollary}

\begin{proof}
To show~(a), we note that 
$$
J_{\OO_i} + I(\BO(\Hbar)) \;=\; 
(I(Z_\OO(\Hbar)) \cap \langle \det(T'_i) \rangle) + I(\BO(\Hbar))
\;\subseteq\; I(Z_\OO(\Hbar)) 
$$
for $i=1,\dots,\ell$, and that we have $\BO(\H)=
\BO(\Hbar) \setminus \mathcal{Z}\bigl(
I(Z_\OO(\Hbar)) \bigr)$.

Now we prove~(b). For every $K$-rational point~$\Gamma$ 
of~$\BO(\H)$, we consider the order ideal
$\OO_\sigma(I_\Gamma) = \mathbb{T}^n \setminus \LT_\sigma(I_\Gamma)$
for some degree compatible term ordering~$\sigma$.
Then the point~$\Gamma$ is contained in 
$\BO^{\rm dfb}(\OO_\sigma(I_\Gamma))$, and the claim follows.
\end{proof}

To complement this discussion, we apply Algorithm~\ref{alg:DFB} to
the setting of Example~\ref{1-x-y-x2-x3}.

\begin{example}\label{1-x-y-x2-x3-cont2}
Let~$\OO$ be the order ideal $\OO=\{1,x,y,x^2,x^3\}$
in~$\mathbb{T}^2$, and let $\H=(1,3,5,5,\dots)$.
In order to cover $\BO(\H)$ with DFB-subschemes, we have
to construct all order ideals with affine Hilbert function~$\H$.
They are $\OO'_1 = \{1,x,y,x^2,xy\}$, $\OO'_2 = \{1,x,y,x^2,y^2\}$,
and $\OO'_3 = \{1,x,y,xy,y^2\}$.
\begin{enumerate}
\item[(a)] When we apply Algorithm~\ref{alg:DFB} to~$\OO'_1$,
we get $J_{\OO'_1}=\langle c_{51}\rangle + I(\BO(\Hbar))$. This is
consistent with the observation that we can exchange the element~$x^3$
of~$\OO$ with the element $xy$ of~$\OO'_1$ in the basis of a ring 
$R_{\X_\Gamma}$ if an only if the entry~$c_{51}$ in~$\Gamma$
is non-zero.

\item[(b)] An application of Algorithm~\ref{alg:DFB}
to~$\OO'_2$ yields $J_{\OO'_2}=\langle c_{52} \rangle + I(\BO(\Hbar))$.

\item[(c)] By applying Algorithm~\ref{alg:DFB} to~$\OO'_3$, 
we get $J_{\OO'_3}=\langle c_{42}c_{51} - c_{41}c_{52} \rangle 
+ I(\BO(\Hbar))$.
\end{enumerate}
\end{example}

Our next task is the following. Suppose we are given a 
function ${\tt Locus}(\OO')$ which returns an 
ideal~$I_{\OO'}$ in~$K[C']$ such that $\mathbb{B}^{\rm df}_{\OO'}
\setminus \mathcal{Z}(I_{\OO'})$ contains exactly the $K$-rational 
points which represent schemes having a degree filtered
$\OO'$-border basis and a certain property~$\mathcal{P}$. 
After covering $\BO(\H)$ with finitely many open subsets
$\BO^{\rm dfb}(\OO')$, we want to
compute an ideal~$J$ in~$K[C]$ such that $\BO(\H) \setminus
\mathcal{Z}(J)$ contains exactly the $K$-rational points
of~$\BO(\H)$ which represent the 0-dimensional schemes
with Hilbert function~$\H$ and property~$\mathcal{P}$.

The first step is to construct some base change formulas
over certain open subsets of~$\BO$.

\begin{proposition}\label{prop:trafo}
Let $\OO=\{t_1,\dots,t_\mu\}$ and $\OO'=\{t'_1,\dots,t'_\mu\}$
be two order ideals such that there is a 0-dimensional
subscheme~$\X$ of~$\mathbb{A}^n_K$ which has both an $\OO$-border 
basis and an $\OO'$-border basis.
\begin{enumerate}
\item[(a)] For $i\in\{1,\dots,\mu\}$, let $T'_i$
be the first column of $t'_i(\A_1,\dots,\A_n)$,
where $\A_1,\dots,\A_n$ are the generic multiplication
matrices with respect to~$\OO$.
Let $T'\in \Mat_\mu(K[C])$ be the matrix 
with columns $T'_1,\dots,T'_\mu$, and let $\delta=\det(T')$.
Then the open subset $D(\delta)$ of~$\BO$ parametrizes
all $K$-rational points of~$\BO$ which represent
schemes having an $\OO'$-border basis.

\item[(b)] Let $G=\{g_1,\dots,g_\nu\}$ be the generic
$\OO$-border prebasis, and let $U_\OO = \BO[x_1,\dots,x_n]/
\langle G\rangle$ be the universal $\OO$-border basis
family. Then both~$\OO$ and~$\OO'$ are $K[C]_\delta$-bases
of $(U_\OO)_\delta$, and if we consider~$\OO$ and~$\OO'$
as row vectors, we have $\OO'=\OO\cdot T'$.

\item[(c)] In the setting of~(b), let $f\in P$.
Then the matrices of the multiplication by~$f$ with respect to the
$K[C]$-bases $\OO$ and~$\OO'$ on~$(U_\OO)_\delta$ satisfy the formula
$$
\A'_f \;=\; {\textstyle\frac{1}{\delta}}\,
\adj(T') \cdot \A_f\cdot T'
$$
\end{enumerate}
\end{proposition}

\begin{proof}
To prove~(a) and~(b), we note that the matrix $t'_i(\A_1,\dots,\A_n)$
is the multiplication matrix of~$t'_i$ in the basis~$\OO$
of the universal family. By the ordering of the terms, we have
$t_1=t'_1=1$. Hence the first column of $t'_i(\A_1,\dots,\A_n)$
contains the tuple of coordinates for~$t'_i$ in the basis~$\OO$.
Consequently, the tuple $(t'_1,\dots,t'_\mu)$ is a
$K[C]$-basis of~$U_\OO$ at all points where $\det(T')\ne 0$,
and this is precisely the set $D(\delta)$. Moreover, the transformation matrix
from the basis~$\OO$ to the basis~$\OO'$ is given by~$T'$
over the set~$D(\delta)$.

Claim~(c) follows from the base change formula for linear maps
and the observation that $(T')^{-1} = \frac{1}{\delta}\, \adj(T')$.
\end{proof}

The next proposition provides a way to
relate the coordinate systems $C=\{c_{ij}\}$
of~$\BO$ and $C'=\{c'_{ij}\}$ of~$\mathbb{B}_{\OO'}$
when an open subset of the latter is mapped 
to $\BO^{\rm dfb}(\OO')$.

\begin{proposition}\label{cijTrafo}
Let $\OO=\{t_1,\dots,t_\mu\}$ and $\OO'=\{t'_1,\dots,t'_\mu\}$
be two order ideals such that there exists a 0-dimensional
subscheme~$\X$ of~$\mathbb{A}^n_K$ which is represented both by
a point in~$\BO$ and by a point in $\mathbb{B}_{\OO'}$.
Let $C'=\{c'_{ij}\}$ be the set of indeterminates
such that $\mathbb{B}_{\OO'} = K[C']/I(\mathbb{B}_{\OO'})$,
and let $\partial\OO'=\{b'_1,\dots,b'_\ell\}$. 

For $j=1,\dots,\ell$,
let $B'_j=(b'_{1j},\dots,b'_{\mu j})\tr$ be the first column
of $b'_j(\A_1,\dots,\A_n)$. Then the representation
of $c'_{ij}$ in the coordinate system~$C$ is given by the $i$-th
entry of $(T')^{-1}B'_j$ which is of the form $\frac{1}{\delta}
p_{ij}$ with $p_{ij}\in K[C]$ for $i=1,\dots,\mu$
and $j=1,\dots,\ell$.
\end{proposition}

\begin{proof}
For $j=1,\dots,\ell$, consider the coordinate
tuple of~$b'_j$ in the basis~$\OO'$ of $(U_{\OO})_\delta$.
One one hand, it is $(c'_{1j},\dots,c'_{\mu j})\tr$. 
On the other hand, as shown in the proof of the preceding proposition, 
the coordinate tuple of~$b'_j$ in the basis~$\OO$
is~$B'_j$. Consequently, we get $b'_j= \OO\cdot B'_j = \OO'\cdot 
(T')^{-1} B'_j$.
By comparing the coefficients of~$b'_j$
in the basis~$\OO'$, we see that $c'_{ij}$ corresponds to the
$i$-th entry of $(T')^{-1} B'_j$.
Using $(T')^{-1} = \frac{1}{\delta}\, \adj(T')$,
the claim follows.
\end{proof}

Next we tackle the main task described above: how to
transform the output of the function ${\tt Locus}(\OO')$
for various order ideals~$\OO'$
to a single ideal defining the desired locus in $\BO(\H)$.

\begin{algorithm}{\upshape\bf (Combining Loci in a Hilbert 
Stratum)}\label{alg:Combine}\\
Let $\OO=\{t_1,\dots,t_\mu\}$ be an order ideal in~$\mathbb{T}^n$, 
and let~$\H$ be an $(n,\mu)$- admissible 
Hilbert function which dominates $\HFa_\OO$. Suppose that there exists a 
function ${\tt Locus}(\OO')$ which returns, for every order ideal~$\OO'$
with $\HFa_{\OO'}=\H$, an ideal~$J_{\OO'}$ in~$K[C']$ such that 
$\mathbb{B}^{\rm df}_{\OO'} \setminus \mathcal{Z}(J_{\OO'})$ 
contains exactly the $K$-rational points which represent schemes 
having a degree filtered $\OO'$-border basis and a certain 
property~$\mathcal{P}$. 
Consider the following instructions.
\begin{enumerate}
\item[(1)] Using Algorithms~\ref{alg:BOHbar} and~\ref{alg:BOH},
compute the ideals $I(\BO(\Hbar))$ and $I(Z_\OO(\Hbar))$
in~$K[C]$.

\item[(2)] Find the set $\{\OO'_1,\dots,\OO'_\ell\}$ 
of all order ideals~$\OO'_i$ such that $\H = \HFa_{\OO'_i}$.

\item[(3)] For $i=1,\dots,\ell$,
perform the following Steps (4) -- (8).

\item[(4)] Using ${\tt Locus}(\OO'_i)$, calculate the 
ideal $J_{\OO'_i}$ in~$K[C'_i]$, where $C'_i$ is the set of indetermintes
such that $B_{\OO'_i} = K[C'_i] / I(\mathbb{B}_{\OO'_i})$.
Let $J_{\OO'_i}=\langle a_{i1},\dots,a_{i r_i}\rangle$.

\item[(5)] Use Proposition~\ref{prop:trafo}.a
to find the polynomial $\delta_i \in K[C]$ such that $D(\delta_i)$
is the open subset of~$\BO(\Hbar)$ whose $K$-rational points
represent schemes having a degree filtered $\OO'$-border basis. 

\item[(6)] Using Proposition~\ref{cijTrafo},
compute polynomials $p_{\kappa\lambda}\in K[C]$ for $\kappa=1,\dots,\mu$ and
$\lambda=1,\dots,\#\partial \OO'_i$ such that $c'_{\kappa\lambda}$ corresponds
to $\frac{1}{\delta_i}\, p_{\kappa\lambda}$.

\item[(7)] Let $z$ be a new indeterminate. For $j=1,\dots,r_i$, 
compute the homogenization $a^\ast_{ij}\in K[C'_i][z]$ 
of~$a_{ij}$ with respect to~$z$. Then let $\hat{a}_{ij}\in K[C]$ be the
result of substituting $c'_{\kappa\lambda}\mapsto p_{\kappa\lambda}$ 
and $z\mapsto \delta_i$ in~$a^\ast_{ij}$.
Let $J_i=\langle \hat{a}_{i1},\dots,\hat{a}_{ir_i}\rangle$.

\item[(8)] Compute $I^{\rm dfb}_{\OO \OO'_i} = I(Z_\OO(\Hbar)) \cap 
\langle \delta_i\rangle$.

\item[(9)] Return the list of triples $(\OO'_i, I^{\rm dfb}_{\OO \OO'_i},
J_i)$ where $i=1,\dots,\ell$.
\end{enumerate}
This is an algorithm which computes a list of triples 
$(\OO'_i, I^{\rm dfb}_{\OO \OO'_i}, J_i)$ of ideals such that,
for the ideal
$$
J \;=\; {\textstyle\sum\limits_{i=1}^\ell}\; \bigl(
(I^{\rm dfb}_{\OO \OO'_i} + I(\BO(\Hbar))) \;\cap\;
(J_i + I(\BO(\Hbar)))  \bigr)
$$
the set $\BO(\Hbar) \setminus \mathcal{Z}(J)$ contains precisely the 
$K$-rational points of~$\BO(\H)$ which represent schemes having
property~$\mathcal{P}$.
\end{algorithm}

\begin{proof}
Finiteness of the algorithm is clear.
Let us start the correctness proof by looking at the
formula for the resulting ideal~$J$. Clearly, if a $K$-rational
point~$\Gamma$ of~$\BO(\Hbar)$ represents a scheme having 
property~$\mathcal{P}$, it is in the set of all $K$-rational
points of~$\BOdf(\OO')$ with this property for every
order ideal~$\OO'$ for which~$\OO'$ is a degree filtered
$K$-basis of~$R_\Gamma$. Hence~$J$ is the sum of the ideals
defining the corresponding complements in~$\BOdf(\OO')$. In view
of Algorithm~\ref{alg:DFB}, it remains to show that~$J_i+I(\BO(\Hbar))$
defines the correct subset of~$\BO(\Hbar)$.

By the hypothesis, the ideal~$J_{\OO'_i}$ defines the correct locus
in~$\mathbb{B}^{\rm df}_{\OO'_i}$. In Step~(6) we apply the transformation
rule for the elements $c'_{ij}$ given in Proposition~\ref{prop:trafo}.d.
However, when we transform the polynomials $a_{ij}$, we are always
working in the localization $K[C]_{\delta_i}$. Hence we may multiply the
resulting rational functions by a power of~$\delta_i$ without 
changing the zero locus. The minimal power of~$\delta_i$ with which
$a_{ij}\vert_{c'_{\kappa\lambda}\mapsto p_{\kappa\lambda}/\delta_i}$
has to be multiplied to make it a polynomial is computed in Step~(7)
using the homogenization of~$a_{ij}$. Altogether, the ideal
$J_i = \langle \hat{a}_{i1},\dots,\hat{a}_{ir}\rangle$ defines the
same zero locus in $D(\delta_i) \subseteq \BO(\Hbar)$ as the
image of~$J_i$ under $c'_{\kappa\lambda} \mapsto \frac{1}{\delta_i}
p_{\kappa\lambda}$, and the correctness proof of the algorithm is complete.
\end{proof}

Recall that the vanishing ideals we compute need not be radical 
ideals. Using this freedom, the recombination of the different 
ideals~$J_i$ in the preceding algorithm may
profit from the following observation.

\begin{lemma}\label{SumOfInters}
Given a ring~$S$ and ideals $I_1,\dots,I_r,J$ in~$S$, we have
$$
\Rad ( {\textstyle\prod\limits_{i=1}^r} I_i + J ) \;=\; 
\Rad ( {\textstyle\bigcap\limits_{i=1}^r} I_i + J ) \;=\; 
\Rad ( {\textstyle\bigcap\limits_{i=1}^r} (I_i+J) )
$$
\end{lemma}

\begin{proof}
The inclusions ``$\subseteq$" are clear.
Therefore it suffices to show $\bigcap_{i=1}^r (I_i+J) \subseteq 
\Rad( \prod_{i=1}^r I_i +J)$.
Let $f = a_1+b_1 = \cdots = a_r+b_r$ with $a_i\in I_i$ 
and $b_i\in J$ for $i=1,\dots,r$. 
Then we have $f^{\,r} = \prod_{i=1}^r a_i +c$ with $c\in J$, 
and the claim follows.
\end{proof}

With the aid of this lemma, the formula for the ideal~$J$
in the above algorithm may be simplified as follows.

\begin{remark}\label{ImprovedCombine}
In the setting of Algorithm~\ref{alg:Combine},
the potentially expensive computation of intersection
ideals in the computation of~$J$ may be
replaced by calculating
$$
\tilde{J} \;=\; {\textstyle\sum\limits_{i=1}^\ell}
\; (I^{\rm dfb}_{\OO\OO'_i} \cap J_i) \;+\; I(\BO(\H))
\hbox{\quad \rm or\quad}
\hat{J} \;=\; {\textstyle\sum\limits_{i=1}^\ell}
\; (I^{\rm dfb}_{\OO\OO'_i} \cdot J_i) \;+\; I(\BO(\H))
$$
Then the $K$-rational points of $\BO(\H) \setminus 
\mathcal{Z}(J)$ are precisely the points~$\Gamma$ in~$\BO(\H)$ 
which represent those affine schemes~$\X_\Gamma$ that have 
property $\mathcal{P}$.
\end{remark}

\bigbreak
%
%

\section{The Cayley-Bacharach Locus in $\BO(\H)$}
\label{The Cayley-Bacharach Locus in BOH}

In this section we use Algorithm~\ref{alg:CBLocDF}
to find the equations describing the Cayley-Bacharach locus
in the border basis scheme. Since the definition of the Cayley-Bacharach
property involves the regularity index of~$R$, we have to fix the
Hilbert function~$\H$ and work in the subscheme~$\BO(\H)$
of the border basis scheme. Furthermore, the characterization 
of the Cayley-Bacharach property in~\cite{KLR1},
Thm.~4.5, which underlies Algorithm~\ref{alg:CBLocDF},
requires us to fix a degree filtered basis of the coordinate ring. Hence
we work in the various $\mathcal{O}'$-DFB subschemes of~$\BO(\H)$,
where $\OO'$ is a degree filtered order ideal for~$\mathcal{O}$.
The recombination of the individual loci is then achieved using
Algorithm~\ref{alg:Combine}.
Thus the result is of the following algorithm is a list of ideals 
which describes the complement of the Cayley-Bacharach locus in~$\BO(\H)$.

\begin{algorithm}{\upshape\bf (Computing the Cayley-Bacharach
Locus)}\label{alg:CBlocus}\\
Let $\OO=\{t_1,\dots,t_\mu\}$ be an order ideal in~$\mathbb{T}^n$,
let $\H=(H_0,H_1,\dots)$ be an $(n,\mu)$-ad\-mis\-si\-ble Hilbert function
which dominates~$\HFa_\OO$, and let $\rho=\min\{i\ge 0 \mid H_i=\mu\}$.
Consider the following sequence of instructions.
\begin{enumerate}
\item[(1)] Compute the set $\{\OO'_1,\dots,\OO'_\ell\}$
of all order ideals in~$\mathbb{T}^n$ with affine Hilbert 
function~$\H$.

\item[(2)] For $i\in\{1,\dots,\ell\}$, let
${\tt CBLocus}(\OO'_i)$ be the function
obtained by applying Algorithm~\ref{alg:CBLocDF}
to~$\OO'_i$. It yields an ideal $J_{\OO'_i}$
in $K[C'_i]$, where~$C'_i$ is the set of indetermintes
such that $B_{\OO'_i} = K[C'_i] / I(\mathbb{B}_{\OO'_i})$.

\item[(3)] Apply Algorithm~\ref{alg:Combine}
using the function ${\tt CBLocus}(\OO'_i)$ in Step~(4).
Return the resulting list of triples 
$(\OO'_i, I^{\rm dfb}_{\OO \OO'_i},J_i)$ where $i=1,\dots,\ell$.
\end{enumerate}
This is an algorithm which computes a list of triples 
$(\OO'_i, I^{\rm dfb}_{\OO \OO'_i}, J_i)$ such that,
for the ideal
$$
I^{\rm NCB}_\OO \;=\; {\textstyle\sum\limits_{i=1}^\ell}\;
\bigl( (I^{\rm dfb}_{\OO \OO'_i} + I(\BO(\Hbar))) \;\cap\;
(J_i + I(\BO(\Hbar))) \bigr)
$$
the set $\BO(\Hbar) \setminus \mathcal{Z}(I^{\rm NCB}_\OO)$ 
contains precisely the $K$-rational points of~$\BO(\H)$ 
which represent schemes having the Cayley-Bacharach property. 
\end{algorithm}

\begin{proof}
This follows by combining Algorithm~\ref{alg:CBLocDF}
and Algorithm~\ref{alg:Combine}.
\end{proof}

Note that we may apply the improvement offered by
Remark~\ref{ImprovedCombine} to this algorithm.
It is also possible to construct a version which avoids
the introduction of new sets of indeterminates~$C'_i$.
Since this version did not yield worthwhile speed-ups
for the actual implementation, we did not include it here.
Let us illustrate Algorithm~\ref{alg:CBlocus} by applying 
it in the setting of Example~\ref{1-x-y-x2-x3-cont2}.

\begin{example}\label{1-x-y-x2-x3-cont3}
In~$\mathbb{T}^2$, we consider the order ideal $\OO=\{1,x,y,x^2,x^3\}$,
and we use the Hilbert function $\H=(1,3,5,5,\dots)$.
As we saw in Example~\ref{1-x-y-x2-x3-cont2}, the scheme $\BO(\H)$
is covered by three open subschemes $\BO^{\rm dfb}(\OO'_i)$
where $i\in\{1,2,3\}$. Let us compute the Cayley-Bacharach
locus for each scheme $\BO^{\rm dfb}(\OO'_i)$.
\begin{enumerate}
\item[(a)] For the order ideal $\OO'_1 = \{1,x,y,x^2,xy\}$,
the algorithm produces the triple $(\OO'_1,I^{\rm dfb}_{\OO\OO'_1},J_1)$
with $I^{\rm dfb}_{\OO\OO'_1}=\langle c_{51}\rangle$ and
$J_1=\langle 1\rangle$.

\item[(b)] Similarly, for the order ideal $\OO'_2 = \{1,x,y,x^2,y^2\}$,
we obtain $I^{\rm dfb}_{\OO\OO'_2}=\langle c_{52}\rangle$ and
$J_2=\langle 1\rangle$.

\item[(c)] Thirdly, for $\OO'_3 = \{1,x,y,xy,y^2\}$, we get
$I^{\rm dfb}_{\OO\OO'_3}=\langle c_{42}c_{51}-c_{41}c_{52}\rangle$ and
$J_3=\langle 1\rangle$.
\end{enumerate}
Using Remark~\ref{ImprovedCombine}, the combined result is
$$
\tilde{I}^{\rm NCB}_\OO \;=\;
{\textstyle\sum\limits_{i=1}^3} (I^{\rm dfb}_{\OO\OO'_i} \cap J_{\OO'_i}) 
+ I(\BO(\Hbar)) \;=\; \langle c_{51},c_{52}\rangle + I(\BO(\Hbar)) \;=\; 
I(Z_\OO(\Hbar))
$$
i.e., an ideal defining the complement of~$\BO(\H)$ in~$\BO(\Hbar)$. 
This is in agreement with the observation that, for $\H=(1,3,5,5,\dots)$,
every $K$-rational point of~$\BO(\H)$ represents a Cayley-Bacharach scheme.
\end{example}

Next we apply Algorithm~\ref{alg:CBlocus} to a case where
not every scheme has the Cayley-Bacharach property.

\begin{example}
Let us consider the order ideal $\OO=\{1,x,x^2,x^3\}$ in~$\mathbb{T}^2$
and $\H =(1,3,4,4,\dots)$. In this case we have three order ideals
with affine Hilbert function~$\H$, namely $\OO'_1 =\{1,x,y,x^2\}$, 
$\OO'_2=\{1,x,y,xy\}$, and $\OO'_3 =\{1,x,y,y^2\}$.
Here Algorithm~\ref{alg:CBlocus} yields the following three triples
$(\OO'_i, I^{\rm dfb}_{\OO \OO'_i}, J_i)$.
\begin{enumerate}
\item[(a)] $I^{\rm dfb}_{\OO\OO'_1} = \langle c_{41} \rangle$
and $J_1=\langle h\rangle$, where
\begin{align*}
h=\; &c_{34}^2 c_{41}^4  -2 c_{31} c_{34} c_{41}^3 c_{44} 
+c_{31}^2 c_{41}^2 c_{44}^2  -2 c_{31}^2 c_{34} c_{41}^2  
+2 c_{21} c_{34} c_{41}^3  +2 c_{31}^3 c_{41} c_{44} 
\\ &
-2 c_{21} c_{31} c_{41}^2 c_{44} +c_{31}^4  -2 c_{21} c_{31}^2 c_{41} 
+c_{21}^2 c_{41}^2  -c_{21} c_{32} c_{41}^2  
-c_{31} c_{33} c_{41}^2  -c_{35} c_{41}^3  
\\ & 
+c_{21} c_{31} c_{41} c_{42} +c_{31}^2 c_{41} c_{43} 
+c_{31} c_{41}^2 c_{45},
\end{align*}

\item[(b)] $I^{\rm dfb}_{\OO\OO'_2} = \langle 
c_{34} c_{41}^2  -c_{31} c_{41} c_{44} -c_{31}^2  +c_{21} c_{41} \rangle$
and $J_2=\langle h\rangle$,

\item[(c)] $I^{\rm dfb}_{\OO\OO'_3} = \langle 
c_{21} c_{32} c_{41} +c_{31} c_{33} c_{41} +c_{35} c_{41}^2  
-c_{21} c_{31} c_{42} -c_{31}^2 c_{43} -c_{31} c_{41} c_{45} \rangle$
and $J_3=\langle h\rangle$.
\end{enumerate}
Using Remark~\ref{ImprovedCombine} and computing up to radical, we get
$$
\tilde{I}^{\rm NCB}_\OO \;=\;
{\textstyle\sum\limits_{i=1}^3} (I^{\rm dfb}_{\OO\OO'_i} \cap \langle h\rangle) 
+ I(\BO(\Hbar)) \equiv \langle c_{31},c_{41}\rangle \cap \langle h\rangle + I(\BO(\Hbar))
$$	
Since $I(Z_\OO(\Hbar)) = \langle c_{31},c_{41}\rangle + I(\BO(\Hbar))$, 
it follows that the Cayley-Bacharach locus in $\BO(\H)$ is defined by a single
polynomial. Using the fact that we can replace~$h$ by its normal
form with respect to~$I(\BO(\Hbar))$, we see that we can also 
use the polynomial
$$
\tilde{h}=-c_{35} c_{41}^3  +c_{21} c_{41} c_{42}^2  
+c_{42}^4  -c_{21} c_{41}^2 c_{43} +c_{31} c_{41} c_{42} c_{43} 
-2 c_{41} c_{42}^2 c_{43} +c_{41}^2 c_{43}^2
$$
to define the Cayley-Bacharach locus in~$\BO(\H)$.
\end{example}

If we combine the locally Gorenstein property and the Cayley-Bacharach
property, we can compute the corresponding locus by intersecting
the loci computed in Algorithm~\ref{alg:GorLoc} and Algorithm~\ref{alg:CBlocus}. 
Another approach, based on~\cite{KLR1}, Algorithm~5.9,
produces a more direct method. Since the resulting algorithm is similar to
Algorithm~\ref{alg:CBlocus}, we did not include it here.

Another application of Algorithm~\ref{alg:CBlocus} is
the following method for calculating the strict Gorenstein
locus in~$\BO(\H)$. It is based on the fact that a 0-dimensional 
affine scheme~$\X$ is strictly Gorenstein if and only if its 
affine Hilbert function is symmetric and it is Cayley-Bacharach 
(see~\cite{KLR1}, Thm.~6.8).

\begin{corollary}{\upshape\bf (Computing the Strict Gorenstein
Locus)}\label{alg:SGorInBOH}\\
In the above setting, the following instructions define
an algorithm which computes an ideal~$I^{\rm SG}_{\OO}$ in~$K[C]$ such
that the $K\!$-rational points~$\Gamma$ of $\BO(\H) \setminus
\mathcal{Z}(I^{\rm SG}_\OO)$ are precisely the ones which
represent 0-dimensional affine schemes~$\X_\Gamma$
that are strictly Gorenstein.
\begin{enumerate}
\item[(1)] Check if the affine Hilbert function~$\H$ is symmetric.
If this is not the case, return the zero ideal and stop.

\item[(2)] Using Algorithm~\ref{alg:CBlocus}, compute an ideal
$I^{\rm NCB}_\OO$ in~$K[C]$ which defines the
non-Cayley-Bacharach locus in~$\BO(\H)$ and return it.
\end{enumerate}
\end{corollary}

\bigbreak
%
%

\section{The Strict Cayley-Bacharach Locus in $\BO(\H)$}
\label{The Strict Cayley-Bacharach Locus in BOH}

Our next goal is to calculate the strict Cayley-Bacharach locus
in~$\BO$. Again we have to fix the Hilbert function first, that is, 
we have to work in a subscheme $\BO(\H)$ of~$\BO$.
Then we use a version of Algorithm~\ref{alg:SCBLocDF}
which is based on a degree filtered border basis.
Finally, we recombine the various loci for all degree filtered 
order ideals. In detail, we have the following algorithm.

\begin{algorithm}{\upshape\bf (Computing the Strict Cayley-Bacharach
Locus)}\label{alg:SCBlocus}\\
Let $\OO=\{t_1,\dots,t_\mu\}$ be an order ideal in~$\mathbb{T}^n$,
let~$\H=(H_0,H_1,\dots)$ be an $(n,\mu)$-admissible Hilbert function which 
dominates $\HFa_\OO$, let $\rho=\min\{ i\ge 0 \mid H_i=\mu \}$, 
and let $\Delta = H_\rho - H_{\rho-1}$.
Consider the following sequence of instructions.
\begin{enumerate}
\item[(1)] Compute the set $\{\OO'_1,\dots,\OO'_\ell\}$
of all order ideals in~$\mathbb{T}^n$ with affine Hilbert 
function~$\H$.

\item[(2)] For $i\in\{1,\dots,\ell\}$, let
${\tt SCBLocus}(\OO'_i)$ be function
obtained by applying Algorithm~\ref{alg:SCBLocDF}
to~$\OO'_i$. It yields an ideal $J_{\OO'_i}$ in $K[C'_i]$,
where~$C'_i$ is the set of indetermintes
such that $B_{\OO'_i} = K[C'_i] / I(\mathbb{B}_{\OO'_i})$.

\item[(3)] Apply Algorithm~\ref{alg:Combine}
using the function ${\tt SCBLocus}(\OO'_i)$ in Step~(4).
Return the resulting list of triples 
$(\OO'_i, I^{\rm dfb}_{\OO \OO'_i},J_i)$, where $i=1,\dots,\ell$.
\end{enumerate}
This is an algorithm which computes a list of triples 
$(\OO'_i, I^{\rm dfb}_{\OO \OO'_i}, J_i)$ such that,
for the ideal
$$
I^{\rm NSCB}_\OO \;=\; {\textstyle\sum\limits_{i=1}^\ell}\;
\bigl( (I^{\rm dfb}_{\OO \OO'_i} + I(\BO(\Hbar))) \;\cap\;
(J_i + I(\BO(\Hbar))) \bigr)
$$
the set $\BO(\Hbar) \setminus \mathcal{Z}(I^{\rm NSCB}_\OO)$ 
contains precisely the $K$-rational points of~$\BO(\H)$ 
which represent schemes having the strict Cayley-Bacharach 
property. 
\end{algorithm}

\begin{proof}
This follows from Algorithm~\ref{alg:SCBLocDF} and
Algorithm~\ref{alg:Combine}.
\end{proof}

To illustrate this algorithm, we apply it in a concrete case.

\begin{example}\label{1-x-y-z-x2-x3}
In~$\mathbb{T}^3$, we consider the order ideal $\OO=\{1,x,y,z,x^2,x^3\}$,
and we let $\H=(1,4,6,6,\dots)$. Then we have $\rho=\Delta=2$,
the ideal $I(\BO(\Hbar))$ is generated by 81 polynomials
\begin{align*}
I(\BO(\Hbar) \;=\; \langle &c_{3\, 10} c_{6 1} -c_{2 8} c_{6 2} +c_{4\, 10} c_{6 2} 
-c_{3 8} c_{6 4} -c_{4 8} c_{6 5} -c_{5 8} c_{6 7} +c_{5\, 10}, \dots,\\
-&c_{1 1} c_{2 1} -c_{1 3} c_{3 1} +c_{1 1} c_{3 3} -c_{1 4} c_{4 1} 
+c_{1 2} c_{4 3} -c_{1 6} c_{5 1} -c_{1 9} c_{6 1} +c_{1 8} c_{6 3}\rangle
\end{align*}
and the ideal $I(Z_\OO(\Hbar))$ is generated by 61 polynomials
$$
I(Z_\OO(\Hbar)) \;=\; \langle c_{6 5}, c_{6 4}, \dots, -c_{1 1} c_{2 1} 
-c_{1 3} c_{3 1} +c_{1 1} c_{3 3} -c_{1 4} c_{4 1} +c_{1 2} c_{4 3} 
-c_{1 6} c_{6 6} \rangle
$$

Notice that there are 15 order ideals
with affine Hilbert function~$\H$, namely $\OO'_1 = \{1,x,y,z,x^2,xy\}$,
$\OO'_2 = \{1,x,y,z,x^2,y^2\}$, $\dots$, $\OO'_{15} = 
\{1,x,y,z,yz,z^2\}$.
Hence Algorithm~\ref{alg:SCBlocus} yields 15 triples
$(\OO'_i, I^{\rm dfb}_{\OO \OO'_i}, J_i)$, where
\begin{enumerate}
\item[(1)] $I^{\rm dfb}_{\OO \OO'_1} = I(Z_\OO(\Hbar)) \cap 
\langle c_{61}\rangle$ and
\begin{align*}
J_1 = \langle
&c_{5 2} c_{6 2} c_{6 3} -c_{5 2} c_{6 1} c_{6 4} -c_{5 1} c_{6 2} c_{6 4} 
+c_{5 1} c_{6 1} c_{6 5} +c_{6 4}^2  -c_{6 3} c_{6 5}, \\
& c_{5 2} c_{6 1} c_{6 2} -c_{5 1} c_{6 2}^2  +c_{6 2} c_{6 4} 
-c_{6 1} c_{6 5},\\
&  -c_{5 5} c_{6 1}^2  +c_{5 4} c_{6 1} c_{6 2} -c_{5 1} c_{6 2} c_{6 4} 
+c_{5 1} c_{6 1} c_{6 5} +c_{6 4}^2  -c_{6 3} c_{6 5},\\
&  c_{5 4} c_{6 1}^2  -c_{5 3} c_{6 1} c_{6 2} +c_{5 1} c_{6 2} c_{6 3} 
-c_{5 1} c_{6 1} c_{6 4},\\
& c_{5 2} c_{6 1}^2  -c_{5 1} c_{6 1} c_{6 2} +c_{6 2} c_{6 3} 
-c_{6 1} c_{6 4}, \\
& -c_{5 1} c_{5 2} c_{6 1} c_{6 3} +c_{5 1}^2 c_{6 2} c_{6 3} 
+c_{5 4} c_{6 1} c_{6 3} -c_{5 3} c_{6 2} c_{6 3},\\
& c_{5 1} c_{5 4} c_{6 1} c_{6 2}\!-\!c_{5 1} c_{5 3} c_{6 2}^2  
\!+\!c_{5 1} c_{5 2} c_{6 1} c_{6 4} -c_{5 1}^2 c_{6 1} c_{6 5} \!
-\!c_{5 4} c_{6 1} c_{6 4} \!+\!c_{5 3} c_{6 2} c_{6 4} \rangle
\end{align*}

\item[$\vdots$\;\;]

\item[(15)] $I^{\rm dfb}_{\OO \OO'_{15}} = I(Z_\OO(\Hbar)) \cap 
\langle c_{5 5} c_{6 4} -c_{5 4} c_{6 5}\rangle$ and
\begin{align*}
J_{15} \;=\; \langle & c_{5 5}^2 c_{6 2} c_{6 3} -c_{5 5}^2 c_{6 1} c_{6 4} 
-c_{5 4} c_{5 5} c_{6 2} c_{6 4} + \cdots  -c_{5 1} c_{5 4} c_{6 5}^2,
\quad \dots,\\
& c_{5 4}^3 c_{6 1} c_{6 2} c_{6 4} -c_{5 3} c_{5 4}^2 c_{6 2}^2 c_{6 4} 
-c_{5 2} c_{5 4}^2 c_{6 1} c_{6 4}^2 + \cdots +c_{5 3} c_{5 4} c_{6 4}^2 c_{6 5}
\rangle
\end{align*}
where $J_{15}$ has 7 generators of degree 4 and 3 generators of degree 6.
\end{enumerate}
Let us compare this result to the ideals defining the Cayley-Bacharach locus
in~$\BO(\H)$ which we compute using Algorithm~\ref{alg:CBlocus}.
\begin{enumerate}
\item[(1)] For $\OO'_1 = \{1,x,y,z,x^2,xy\}$ the ideal $\tilde{J}_1$
is generated by 58 polynomials 
\begin{align*}
\tilde{J}_1 =\langle\; &
c_{5 2} c_{6 2} c_{6 3} -c_{5 2} c_{6 1} c_{6 4} -c_{5 1} c_{6 2} c_{6 4} 
+c_{5 1} c_{6 1} c_{6 5} +c_{6 4}^2  -c_{6 3} c_{6 5},\quad \dots, \\ &
c_{3 1} c_{3 3} c_{4 1}^2 c_{5 3}^2 c_{6 1} c_{6 2}^4 c_{6 3} 
+c_{3 3} c_{4 1}^3 c_{5 3}^2 c_{6 2}^5 c_{6 3}+\cdots 
-\tfrac{99}{4} c_{2 3} c_{5 3}^2 c_{6 3}^2 c_{6 5}
\; \rangle
\end{align*}

\item[$\vdots$\;\;]

\item[(15)] For $\OO'_{15} = \{1,x,y,z,yz,z^2\}$ the ideal
$\tilde{J}_{15}$ is generated by 54 polynomials  
\begin{align*}
\tilde{J}_{15} = \langle &
c_{5 5}^2 c_{6 2} c_{6 3} -c_{5 5}^2 c_{6 1} c_{6 4} 
-c_{5 4} c_{5 5} c_{6 2} c_{6 4} +\cdots -c_{5 1} c_{5 4} c_{6 5}^2,
\quad\dots,\\
&c_{2 4}^2 c_{3 4}^2 c_{5 2} c_{5 3}^2 c_{5 4} c_{5 5}^2 c_{6 2}^4 c_{6 4}^2  
-c_{2 4}^2 c_{3 4}^2 c_{5 1} c_{5 3} c_{5 4}^2 c_{5 5}^2 c_{6 2}^4 c_{6 4}^2  
+\cdots\\
&+\tfrac{515}{2} c_{2 4}^3 c_{5 3}^2 c_{5 4} c_{6 4} c_{6 5}^5 c_{6 10}
\rangle.
\end{align*}
\end{enumerate}
It is straightforward to check that $J_i \subseteq \tilde{J}_i$ for $i=1,\dots,15$,
i.e., that the strict Cayley-Bacharach locus is contained in the Cayley-Bacharach
locus of~$\BO(\H)$. 

To see that these two loci differ, we construct one point
in $\BO(\Hbar) \setminus \mathcal{Z}(I^{\rm NCB}_\OO)$ which is not contained
in $\BO(\Hbar) \setminus \mathcal{Z}(I^{\rm NSCB}_\OO)$. 
For this purpose, consider the point $\Gamma=(\gamma_{ij})$ such that $\gamma_{32}=
\gamma_{61}=\gamma_{67}=1$ and $\gamma_{ij}=0$ otherwise. It represents the scheme~$\X$
in~$\AA^3_K$ defined by $I_\X=\langle xy-x^3,\, xz-y,\, y^2,\, yz,\, z^2,\, x^2y,\,
x^2z-x^3,\, x^4,\, x^3y,\, x^3z \rangle$.
It is easy to check that~$\Gamma$ is contained in $\mathcal{Z}(J_i)$ for
$i=1,\dots,15$ and in $\mathcal{Z}(I(\BO(\Hbar)))$. Consequently,
the point~$\Gamma$ is a zero of~$I^{\rm NSCB}_\OO$, i.e., the scheme~$\X$
does not have the strict Cayley-Bacharach property.
On the other hand, the point~$\Gamma$ is not a zero of
$f = c_{3 2} c_{6 1}^2  -c_{3 1} c_{6 1} c_{6 2} 
+c_{4 2} c_{6 1} c_{6 2} -c_{4 1} c_{6 2}^2$, and it is straightforward to vertify 
that $f \in \tilde{J}_1\setminus J_1$. Hence $\Gamma$ is not contained in
$\mathcal{Z}(I^{\rm NCB}_\OO)$, i.e., the scheme~$\X$ has the Cayley-Bacharach 
property.
\end{example}

\begin{remark}\label{StrictCB=StrictGor}
Recall that the strict Gorenstein locus equals the
strict Cayley-Bacharach locus in~$\BO(\H)$ if $\Delta=1$, and it is empty 
otherwise (see~\cite{KLR1}, Thm.~6.12). Thus
Algorithm~\ref{alg:SCBlocus} yields an alternative
to Corollary~\ref{cor:StrictGorLoc} for computing the
strict Gorenstein locus in~$\BO(\H)$.
\end{remark}

\bigbreak
%
%

\section{The Strict Complete Intersection Locus in $\BO(\H)$}
\label{The Strict Complete Intersection Locus in BOH}

The last case that we consider is the
locus in $\BO(\H)$ whose $K$-rational points
represent 0-dimensional affine schemes which are
strict complete intersections. Clearly, the first
necessary condition is that the affine Hilbert function~$\H$
has to be symmetric. Then we cover $\BO(\H)$
by $\OO'$-DFB subschemes and use a suitable
version of Algorithm~\ref{alg:SCILocDF} in~$\BO^{\rm dfb}(\OO')$.
Lastly, we use Algorthm~\ref{alg:Combine} to recombine the
results. In this way we arrive at the following algorithm.

\begin{algorithm}{\upshape\bf (Computing the Strict Complete Intersection
Locus)}\label{alg:SCIlocus}\\
Let $\OO=\{t_1,\dots,t_\mu\}$ be an order ideal in~$\mathbb{T}^n$,
let~$\H=(H_0,H_1,\dots)$ be an $(n,\mu)$-admissible Hilbert function which 
dominates $\HFa_\OO$, let $\rho=\min\{ i\ge 0 \mid H_i=\mu \}$, 
and let $\Delta = H_\rho - H_{\rho-1}$.
Consider the following sequence of instructions.
\begin{enumerate}
\item[(1)] Check if~$\H$ is symmetric. If this is not the case,
return the triple $(\OO,I(B_\OO(\Hbar)),\langle 0\rangle)$ and stop.

\item[(2)] Compute the set $\{\OO'_1,\dots,\OO'_\ell\}$
of all order ideals in~$\mathbb{T}^n$ with affine Hilbert 
function~$\H$.

\item[(3)] For $i\in\{1,\dots,\ell\}$, let
${\tt SCILocus}(\OO'_i)$ be function
obtained by applying Algorithm~\ref{alg:SCILocDF}
to~$\OO'_i$. It yields an ideal $J_{\OO'_i}$ in $K[C'_i]$,
where~$C'_i$ is the set of indetermintes
such that $B_{\OO'_i} = K[C'_i] / I(\mathbb{B}_{\OO'_i})$.

\item[(4)] Apply Algorithm~\ref{alg:Combine}
using the function ${\tt SCILocus}(\OO'_i)$ in Step~(4).
Return the resulting list of triples 
$(\OO'_i, I^{\rm dfb}_{\OO \OO'_i},J_i)$, where $i=1,\dots,\ell$.
\end{enumerate}
This is an algorithm which computes a list of triples 
$(\OO'_i, I^{\rm dfb}_{\OO \OO'_i}, J_i)$ such that,
for the ideal
$$
I^{\rm NSCI}_\OO \;=\; {\textstyle\sum\limits_{i=1}^\ell}\;
\bigl( (I^{\rm dfb}_{\OO \OO'_i} + I(\BO(\Hbar))) \;\cap\;
(J_i + I(\BO(\Hbar))) \bigr)
$$
the set $\BO(\Hbar) \setminus \mathcal{Z}(I^{\rm NSCI}_\OO)$ 
contains precisely the $K$-rational points of~$\BO(\H)$ 
which represent schemes having the strict complete intersection
property. 
\end{algorithm}

\begin{proof}
Since the affine Hilbert function~$\H$ of a 0-dimensional strict
complete intersection is symmetric, Step~(1) returns the correct result
when~$\H$ is not symmetric.
The remaining claims follow from Algorithm~\ref{alg:SCILocDF} and
Algorithm~\ref{alg:Combine}.
\end{proof}

Let us use this algorithm to compute the strict complete intersection
locus in a simple example.

\begin{example}\label{1-x-y-x2-x3-x4}
Consider the order ideal $\OO=\{1,x,y,x^2,x^3,x^4\}$ in~$\mathbb{T}^2$ and
the Hilbert function $\H=(1,3,5,6,6,\dots)$ which dominates
$\HFa_\OO = (1,3,4,5,6,6,\dots)$. We apply Algorithm~\ref{alg:SCIlocus} 
and note the main results.
\begin{enumerate}
\item[(1)] The difference function of~$\H$ is $(1,2,2,1,0,\dots)$ and thus
symmetric.

\item[(2)] We find six order ideals, namely $\OO'_1 = 
\{ 1,x,y,x^2,xy,x^3\}$, $\OO'_2 = \{ 1,x,y,x^2,xy,x^2y\}$,
$\OO'_3 = \{ 1,x,y,x^2,y^2,x^3\}$, $\OO'_4 = \{ 1,x,y,x^2,y^2,y^3\}$,
$\OO'_5 = \{ 1,x,y,xy,y^2,xy^2\}$, and $\OO'_6 = \{ 1,x,y,xy,y^2,y^3\}$.

\item[(3)] Using ${\tt SCILocus}(\OO'_i)$ for $i=1,\dots,6$, we get
the following results:
\begin{itemize}
\item[(3.1)]  $I^{\rm dfb}_{\OO \OO'_1} = \langle\; c_{61} \;\rangle$ and
$ J_1 = \langle\;
c_{5 5}^2 c_{6 1}^4  -2 c_{5 1} c_{5 5} c_{6 1}^3 c_{6 5} 
+c_{5 1}^2 c_{6 1}^2 c_{6 5}^2+\cdots -c_{4 2} c_{6 1}^2, \allowbreak 
c_{5 5}^2 c_{6 1}^3 c_{6 2} -2 c_{5 1} c_{5 5} c_{6 1}^2 c_{6 2} c_{6 5}
+c_{5 1}^2 c_{6 1} c_{6 2} c_{6 5}^2+\cdots -c_{4 2} c_{6 1} c_{6 2}
\;\rangle$

\item[(3.2)]  $I^{\rm dfb}_{\OO \OO'_2} = 
\langle\, c_{5 5} c_{6 1}^2  -c_{5 1} c_{6 1} c_{6 5} -c_{5 1}^2  
+c_{4 1} c_{6 1} \,\rangle$ and
$ J_2 = \langle\; 
c_{5 5}^2 c_{6 1}^4  -2 c_{5 1} c_{5 5} c_{6 1}^3 c_{6 5}$ $ 
+c_{5 1}^2 c_{6 1}^2 c_{6 5}^2+\cdots  -c_{4 2} c_{6 1}^2,
c_{4 5} c_{5 2} c_{5 5} c_{6 1}^4  
-c_{4 5} c_{5 1} c_{5 5} c_{6 1}^3 c_{6 2} 
+\cdots -c_{4 1} c_{4 2} c_{6 1} c_{6 2}
\;\rangle$

\item[$\vdots$\;\;]

\item[(3.6)]  $I^{\rm dfb}_{\OO \OO'_6} = 
\langle\, c_{4 2} c_{4 3} c_{5 2} c_{6 1} +c_{4 4} c_{5 2}^2 c_{6 1} 
-\cdots -c_{4 1} c_{5 2} c_{6 2} c_{6 6} 
\,\rangle$ and
$ J_6 = \langle\,
c_{4 5}^2 c_{5 2}^2 c_{6 1}^2 c_{6 2}^2$ $  
-2 c_{4 2} c_{4 5} c_{5 2} c_{5 5} c_{6 1}^2 c_{6 2}^2 +\cdots
-c_{2 2} c_{5 1} c_{5 2} c_{6 2}^2 c_{6 6},
c_{4 5}^2 c_{5 2}^2 c_{6 1}^3 c_{6 2} 
- \cdots -c_{2 1} c_{5 1} c_{5 2} c_{6 2}^2 c_{6 6}
\,\rangle$
\end{itemize}
\end{enumerate}
Thus we obtain 6 triples $(\OO'_i, I^{\rm dfb}_{\OO \OO'_i}, J_i)$,
and using Remark~\ref{ImprovedCombine}, we get that
$$
I^{\rm NSCI}_\OO = {\textstyle\sum\limits_{i=1}^6}
I^{\rm dfb}_{\OO \OO'_i}\cdot J_i + I(\BO(\Hbar))
$$
defines the complement of the strict complete intersection
locus in~$\BO(\Hbar)$. Here
$\sum_{i=1}^6 I^{\rm dfb}_{\OO \OO'_i}\cdot J_i$
is generated by 14 polynomials, and the ideal $I(\BO(\Hbar))$
is generated by 30 polynomials 
\begin{align*}
I(\BO(\Hbar)) =\langle\; &
c_{5 2} c_{6 1} -c_{5 1} c_{6 2},\;
c_{3 4} c_{6 1} +c_{6 4} c_{6 5} +c_{5 4} -c_{6 6},
\quad \dots,
\\ &
c_{1 1} c_{2 1} +c_{1 2} c_{3 1} -c_{1 1} c_{3 2} +c_{1 3} c_{4 1} 
+c_{1 4} c_{5 1} +c_{1 6} c_{6 1} -c_{1 5} c_{6 2}
\;\rangle
\end{align*}

This example also shows that even if an order ideal~$\OO$ does not
have a symmetric Hilbert function, there may exist a strict complete 
intersection scheme~$\X_\Gamma$ represented by a $K$-rational point~$\Gamma$ 
in~$\BO$. Clearly, this point~$\Gamma$ is then not contained in~$\BOdf$. 

In more detail, notice that $\HFa_\OO$ is not symmetric, but the affine Hilbert
function~$\H$, which dominates~$\HFa_\OO$, is symmetric.
Let us consider the 0-dimensional subscheme~$\X$ of~$\AA^2_K$ defined
by~$I_\X= \langle x^3-xy,\, y^2\rangle$. Then the degree forms $x^3,\, y^2$
are a homogeneous regular sequence, and
hence~$\X$ is a strict complete intersection scheme.
For the order ideal $\OO=\{1,x,y,x^2,y^2,x^3,x^4\}$,
the ideal~$I_\X$ has an $\OO$-border basis, namely
$$
G \;=\; \{\; y^2,\, x^3-xy,\, xy^2,\, x^3y-x^2y,\, x^2y^2  \;\}
$$
Thus the point~$\Gamma$ of~$\BO$ which represents~$\X$ is contained
in the strict complete intersection locus of~$\BO(\H)$.
However, the ideal~$\DF(I_\X)=\langle x^3,y^2\rangle$ has no
$\OO$-border basis, since $x^3\in\OO$. Thus the point~$\Gamma$
is not contained in~$\BOdf$.
\end{example}


\bigbreak
\section*{Acknowledgements}

The first and second authors were partially supported by the 
Vietnam National Foundation for Science and Technology Development
(NAFOSTED) grant number 101.04-2019.07. The third author would
like to thank the University of Passau for its hospitality and support
during part of the preparation of this paper. The authors are indebted
to the referee for several insightful and useful remarks.

\bigbreak

\end{document}